\documentclass[final,reqno]{elsarticle}
\usepackage{mathrsfs}
\usepackage{amssymb,amsmath,amsfonts,latexsym}
\usepackage{amsthm}
\usepackage{graphicx,float,epsfig,color,fancyhdr}
\usepackage{framed}

\newtheorem{lemma}{Lemma}[section]

\newtheorem{proposition}{Proposition}[section]
\newtheorem{theorem}{Theorem}[section]

\def\CE{\mathcal{E}}
\def\CM{\mathcal{X}}
\def\CN{\mathcal{Y}}

\def\CT{\mathcal{T}}
\def\E{K}
\def\G{\Gamma}

\def\Hcd{H_0^2(\O)}

\def\LO{L^2(\O)}
\def\N{\mathbb{N}}
\def\O{\Omega}
\def\P{\mathbb{P}}
\def\R{\mathbb{R}}
\def\dim{\mathop{\mathrm{\,dim}}\nolimits}

\def\l{\lambda}
\def\PiK{\Pi_{\E}^{\Delta}}

\def\sp{\mathop{\mathrm{sp}}\nolimits}

\def\VK{V^{\E}_h}

\def\HdO{{H^2(\O)}}
\def\qen{\qquad\text{in }} 

\def\qqpt{\qquad\forall}
\def\qso{\qquad\text{on }} 
 
\def\o{\omega}
\def\dn{\partial_n}
\def\HdoO{{H_0^2(\O)}}

\def\HdsO{{H^{2+s}(\O)}}
\def\HdoK{{H^{2}(\E)}}
\def\WK{W^{\E}_h}
\def\Wh{W_h}
\def\PioK{\Pi_{\E}^{0}}

\allowdisplaybreaks

\begin{document}
\begin{frontmatter}

\title{A virtual element method for
the vibration problem \\
of Kirchhoff plates}

\author[1,2]{David Mora}
\ead{dmora@ubiobio.cl}
\address[1]{GIMNAP, Departamento de Matem\'atica,
Universidad del B\'io-B\'io, Casilla 5-C, Concepci\'on, Chile.}
\author[2,3]{Gonzalo Rivera}
\ead{grivera@ing-mat.udec.cl}
\address[2]{Centro de Investigaci\'on en Ingenier\'ia Matem\'atica
(CI$^2$MA), Universidad de Concepci\'on, Concepci\'on, Chile.}
\author[2,3]{Iv\'an Vel\'asquez}
\ead{ivelasquez@ing-mat.udec.cl}
\address[3]{Departamento de Ingenier\'ia Matem\'atica,
Universidad de Concepci\'on, Concepci\'on, Chile.}

\begin{abstract} 
The aim of this paper is to develop a virtual element method (VEM) for the
vibration problem of thin plates on polygonal meshes.
We consider a variational formulation relying only on the
transverse displacement of the plate and propose
an $H^2(\Omega)$ conforming discretization by means of the VEM
which is simple in terms of degrees of freedom and coding aspects.
Under standard assumptions on the computational
domain, we establish that the resulting scheme
provides a correct approximation of the spectrum
and prove optimal order error estimates for the
eigenfunctions and a double order for the eigenvalues.
The analysis restricts to simply connected polygonal
clamped plates, not necessarily convex.
Finally, we report several numerical experiments
illustrating the behaviour of the proposed scheme
and confirming our theoretical results on different families of meshes.
Additional examples of cases not covered by our theory are also presented.
\end{abstract}


\end{frontmatter}


\setcounter{equation}{0}
\section{Introduction}
\label{SEC:INTR}

The {\it Virtual Element Method} (VEM), introduced in \cite{BBCMMR2013,BBMR2014},
is a recent generalization of the Finite Element Method
which is characterized by the capability of dealing with
very general polygonal/polyhedral meshes.
The interest in numerical methods that can make use of general
polytopal meshes has recently undergone a significant growth
in the mathematical and engineering literature; among
the large number of papers on this subject, we cite as a minimal sample
\cite{BBCMMR2013,CGH14,DPECMAME2015,DPECRAS2015,ST04,TPPM10}.

Indeed, polytopal meshes can be very useful for
a wide range of reasons, including meshing of
the domain (such as cracks) and data (such as inclusions) features,
automatic use of hanging nodes, use of moving meshes, adaptivity.
Moreover, the VEM presents the advantage to easily implement
highly regular discrete spaces.
Indeed, by avoiding the explicit construction of the local basis
functions, the VEM can easily handle general polygons/polyhedrons
without complex integrations on the element
(see \cite{BBMR2014} for details on the coding aspects of the method).
The Virtual Element Method has recently been applied successfully 
to a wide range of problems, see for instance
\cite{AABMR13,ABMVsinum14,BBCMMR2013,BBMR2014,BLM2015,BMRR,BBPS2014,BM12,CG2016,CMS2016,CL,Paulino-VEM,GVXX,MRR2015,PG2015,PPR15,vacca1,vacca2,WRR}.

The numerical approximation of eigenvalue problems for partial
differential equations derived from engineering applications, is object
of great interest from both, the practical and theoretical points of
view. We refer to \cite{Boffi,BGG2012} and the references therein for
the state of the art in this subject area. In particular, this paper
focus on the so called thin plate vibration problem, which involves the
biharmonic operator.
Among the existing techniques to solve this problem, various finite
element methods have been introduced and analyzed.
In particular, we mention nonconforming methods
and different mixed formulations for the Kirchhoff model,
see for instance \cite{BO,Ca,MORR,MoRo2009,Ra}.
More recently, in \cite{BMS2015} a discontinuous Galerkin method
has been proposed and analyzed for the vibration
and buckling problems of thin plates.
On the other hand, the construction of high regularity conforming
finite elements for $H^2(\O)$ is difficult in general,
since them usually involve a large number of degrees of freedom
(see \cite{ciarlet}).

Recently, thanks to the flexibility of VEM, it has been
presented in \cite{BM13,BM12} that virtual elements
can be used to build global discrete spaces of arbitrary
regularity that are simple in terms of degrees of freedom
and coding aspects (see also \cite{ABSVsinum16,BMR16}).
Thus, in the present contribution,
we follow a similar approach in order to solve
an eigenvalue problem modelling 
the two-dimensional plate vibration problem
considering a conforming $C^1$ discrete formulation.



The aim of this paper is to introduce and analyze a $C^1$-VEM 
which applies to general polygonal meshes, made by possibly
non-convex elements, for the two-dimensional plate vibration problem.
We begin with a variational formulation of the spectral problem relying only
on the transverse displacement of the plate. Then,
we exploit the capability of VEM to built highly regular discrete spaces
and propose a conforming $H^2(\O)$ discrete formulation.
The resulting discrete bilinear form is continuous and elliptic.
This method makes use of a very simple set of degrees of freedom,
namely 3 degrees of freedom per vertex of the mesh.
By using the abstract spectral approximation
theory (see \cite{DNR1,DNR2}), under rather mild assumptions on the polygonal
meshes, we establish that the resulting scheme provides a correct
approximation of the spectrum and prove optimal order error estimates
for the eigenfunctions and a double order for the eigenvalues.
We remark that the present method is new on triangular meshes,
and in this case the computational cost is almost $3N_v$, where
$N_v$ denotes the number of vertices, thus providing
a very competitive alternative in comparison to other classical techniques
based on finite elements.

The outline of this article is as follows: We introduce in
Section~\ref{SEC:STAT} the variational formulation of the vibration
eigenvalue problem, define a solution operator and establish its
spectral characterization. In Section~\ref{SEC:DISCRETE}, we introduce
the virtual element discrete formulation, describe the spectrum of a
discrete solution operator and prove some auxiliary results.
In Section~\ref{SEC:convergence}, we prove
that the numerical scheme provides a correct spectral approximation and
establish optimal order error estimates for the eigenvalues and
eigenfunctions. Several numerical tests that
allow us to assess the convergence properties of the method, to confirm
that it is not polluted with spurious modes and to check whether the
experimental rates of convergence agree with the theoretical ones
are reported in Section~\ref{SEC:NUMER}.
Finally, we summarize some conclusions in Section~\ref{SEC:Conclusion}.

Throughout the article we will use standard notations for Sobolev
spaces, norms and seminorms. Moreover, we will denote by $C$ a generic
constant independent of the mesh parameter $h$, which may take different
values in different occurrences.

Finally, given a linear bounded operator $T:\,X\to X$,
defined on a Hilbert space $X$, we denote its spectrum by 
$\sp(T):=\left\{z\in\mathbb{C}:\,\left(zI-T\right)\mbox{ is not
invertible}\right\}$ and by $\rho(T):=\mathbb{C}\setminus\sp(T)$ the resolvent
set of $T$. Moreover, for any $z\in\rho(T)$,
$R_z(T):=\left(zI-T\right)^{-1}:\,X\to X$ denotes the resolvent operator
of $T$ corresponding to $z$. 


\setcounter{equation}{0}
\section{The spectral problem}
\label{SEC:STAT}

Let $\O\subset\R^2$ be a polygonal bounded simply-connected domain
occupied by the mean surface of a plate, clamped on its whole boundary
$\G$. The plate is assumed to be homogeneous, isotropic, linearly
elastic, and sufficiently thin as to be modeled by Kirchhoff-Love
equations. We denote by $w$ the transverse displacement of the mean
surface of the plate. 

The plate vibration problem reads as follows:

Find $(\l,w)\in\R\times\HdO$, $w\ne0$, such that
\begin{equation}
\label{vibra}
\left\{\begin{array}{ll}
\Delta^2w=\l w&\qen\O,
\\
w=\dn w=0&\qso\G,
\end{array}\right.
\end{equation}
where $\l=\o^2$, with $\o>0$ being the vibration frequency, and $\dn$
denotes the normal derivative. To simplify the notation we have taken
the Young modulus and the density of the plate, both equal to $1$.

To obtain a weak formulation of the spectral problem~\eqref{vibra},
we multiply the corresponding equation by $v\in\HdoO$ and integrate
twice by parts in $\O$. Thus, we obtain:

Find $(\l,w)\in\R\times\HdoO$, $w\ne0$, such that
\begin{equation}
\label{vibrapri}
a(w,v)
=\l b(w,v)
\qqpt v\in\HdoO,
\end{equation}
in \eqref{vibrapri} the bilinear forms are defined
for any $w,v\in\HdoO$ by
\begin{align*}
a(w,v)
& :=\int_{\O}D^2 w:\,D^2 v,
\\
b(w,v)
& :=\int_{\O}wv,
\end{align*}
with $":"$ denotes the usual scalar product of $2\times2$-matrices,
$D^2 v:=(\partial_{ij}v)_{1\le i,j\le2}$ denotes the Hessian
matrix of $v$.
We note that those are bounded bilinear symmetric forms.
Moreover, it is immediate to prove that the eigenvalues
of the problem above are real and positive.

Next, we define the solution operator associated with
the variational eigenvalue problem~\eqref{vibrapri}:
\begin{align*}
T:\ \HdoO & \longrightarrow \HdoO,
\\
f & \longmapsto Tf:=u,
\end{align*}
where $u\in\HdoO$ is the solution of the corresponding source problem:
\begin{equation}
\label{T1}
a(u,v)=b(f,v)
\qquad\forall v\in\HdoO.
\end{equation}

The following lemma allows us to establish the well-posedness of this
source problem.

\begin{lemma}
\label{ha-elipt}
There exists a constant $\alpha>0$, depending on $\O$, such that
$$
a(v,v)
\ge\alpha\left\|v\right\|_{2,\O}^2
\qquad\forall v\in\HdoO.
$$
\end{lemma}

\begin{proof}
The result follows immediately from the fact that
$\Vert D^2 v\Vert_{0,\O}$ is a norm on $\HdoO$,
equivalent with the usual norm. 
\end{proof}

We deduce from Lemma~\ref{ha-elipt} that the linear operator $T$ is well
defined and bounded. Notice that $(\l,w)\in\R\times\HdoO$ solves
problem~\eqref{vibrapri} (and hence problem~\eqref{vibra}) if and only if $Tw=\mu w$
with $\mu\neq0$ and $w\ne0$, in which case $\mu:=\frac1{\l}$.
Moreover, it is easy to check that $T$ is self-adjoint with respect to
the $a(\cdot,\cdot)$ inner product. Indeed, given
$f,g\in\HdoO$, 
$$
a(Tf,g)
=b(f,g)
=b(g,f)
=a(Tg,f)
=a(f,Tg).
$$

The following is an additional regularity result for the solution of
problem~\eqref{T1} and consequently, for the eigenfunctions of $T$.

\begin{lemma}
\label{LEM:REG} 
There exist $s\in(\frac12,1]$ and $C>0$ such that, for all $f\in\LO$,
the solution $u$ of problem~\eqref{T1} satisfies
$u\in\HdsO$ and
$$\Vert u\Vert_{2+s,\O}\le C\Vert f\Vert_{0,\O}.$$
\end{lemma}

\begin{proof}
The proof follows from the classical regularity result for the
biharmonic problem with its right-hand side in $\LO$ (cf. \cite{G}).
\end{proof}

The constant $s$ in the lemma above is the Sobolev regularity for the
biharmonic equation with homogeneous Dirichlet boundary conditions.
This constant only depends on the domain $\O$. If $\O$ is convex,
then $s=1$. Otherwise, the lemma holds for all $s < s_0$,
where $s_0\in(\frac12,1)$ depends on the largest reentrant
angle of $\O$ (see \cite{G} for the precise equation determining $s_0$).
Hence, because of the compact inclusion $\HdsO\hookrightarrow\HdoO$,
$T$ is a compact operator. Therefore, we have the following spectral
characterization result.

\begin{lemma}
\label{CHAR_SP}
The spectrum of $T$ satisfies
$\sp(T)=\{0\}\cup\left\{\mu_k\right\}_{k\in\N}$, where
$\left\{\mu_k\right\}_{k\in\N}$ is a sequence of real positive
eigenvalues which converges to $0$. The multiplicity
of each eigenvalue is finite.
\end{lemma}


\setcounter{equation}{0}
\section{Spectral approximation}
\label{SEC:DISCRETE}

In this section, first we recall the mesh construction and the
assumptions considered to introduce the discrete virtual element spaces.
Then, we will introduce a virtual element discretization of
the eigenvalue problem~\eqref{vibrapri} and provide a spectral
characterization of the resulting discrete eigenvalue problem.

Let $\left\{\CT_h\right\}_h$ be a sequence of decompositions of $\O$
into polygons $\E$. Let $h_\E$ denote the diameter of the element $\E$
and $h$ the maximum of the diameters of all the elements of the mesh,
i.e., $h:=\max_{\E\in\CT_h}h_\E$. In what follows, we denote by $N_K$ the number of vertices of $K$.

For the analysis, we will make the following
assumptions as in \cite{BBCMMR2013,BMR16}:
there exists a positive real number $C_{\CT}$ such that,
for every $h$ and every $\E\in \CT_h$,
\begin{itemize}
\item[{\bf A1}:] the ratio between the shortest edge
and the diameter $h_\E$ of $\E$ is larger than $C_{\CT}$;
\item[{\bf A2}:] $\E\in\CT_h$ is star-shaped with
respect to every point of a  ball
of radius $C_{\CT}h_\E$.
\end{itemize}

For any subset $S\subseteq\R^2$ and nonnegative
integer $k$, we indicate by $\P_{k}(S)$ the space of
polynomials of degree up to $k$ defined on $S$.

We consider now a simple polygon $\E$
(meaning open simply connected sets whose boundary is a non-intersecting line
made of a finite number of straight line segments) and we define
the following finite-dimensional space
\begin{align*}
\VK
:=\left\{v_h\in \HdoK : \Delta^2v_h\in\P_{2}(\E), v_h|_{\partial\E}\in C^0(\partial\E),
v_h|_e\in\P_3(e)\,\,\forall e\in\partial\E,\right.\\
\left.\nabla v_h|_{\partial\E}\in C^0(\partial\E)^2,
\dn v_h|_e\in\P_1(e)\,\,\forall e\in\partial\E\right\},
\end{align*}
where $\Delta^2$ represents the biharmonic operator
and $\dn$ denotes the normal derivative.
We observe
that any $v_{h}\in\VK$ satisfy the following conditions:
\begin{itemize}
\item the trace (and the trace of the gradient)
on the boundary of $\E$ is continuous;

\item $\P_2(\E)\subseteq\VK$.
\end{itemize}

We now introduce two sets ${\bf D_1}$ and ${\bf D_2}$
of linear operators from $\VK$ into $\R$. For
all $v_{h}\in\VK$ they are defined as follows:

\begin{itemize}
\item ${\bf D_1}$ contains linear operators
evaluating $v_{h}$ at the $N_{\E}$ vertices of $\E$;
\item ${\bf D_2}$ contains linear operators evaluating
$\nabla v_h$ at the $N_{\E}$ vertices of $\E$.
\end{itemize}

Note that, as a consequence of definition of $\VK$,
the output values of the two sets of operators ${\bf D_1}$ and ${\bf D_2}$
are sufficient to uniquely determine $v_{h}$
and $\nabla v_h$ on the boundary of $\E$.

In order to construct the discrete scheme, we need some preliminary
definitions. First, we split the bilinear forms $a(\cdot,\cdot)$
and $b(\cdot,\cdot)$ introduced in the previous section as follows:
$$
a(u,v)=\sum_{\E\in\CT_h}a_{\E}(u,v),
\qquad u,v\in\HdoO,
$$
$$
b(u,v)=\sum_{\E\in\CT_h}b_{\E}(u,v),
\qquad u,v\in\HdoO,
$$
with
\begin{equation*}
\label{alocal}
a_{\E}(u,v)
:=\int_{\E}D^2 u:\,D^2 v,
\qquad u,v\in\HdoK,
\end{equation*}
and
\begin{equation*}
\label{blocal}
b_{\E}(u,v)
:=\int_{\E}uv,
\qquad u,v\in\HdoK.
\end{equation*}
Now, we
define the projector $\PiK:\ \VK\longrightarrow\P_2(\E)\subseteq\VK$ for
each $v\in\VK$ as the solution of 
\begin{subequations}
\begin{align}
a_{\E}\big(\PiK v,q\big)
& =a_{\E}(v,q)
\qquad\forall q\in\P_2(\E),
\label{numero}
\\
((\PiK v,q))_{\E}
&=((v,q))_{\E} \qquad\forall q\in\P_1(\E),
\label{numeroo}
\end{align}
\end{subequations}
where $((\cdot,\cdot))_{\E}$ is defined as follows:
\begin{equation*}
((u,v))_{\E}=\sum_{i=1}^{N_{\E}}u(P_i)v(P_i)\qquad\forall u,v\in C^0(\partial\E),
\end{equation*}
where $P_i, 1\le i\le N_{\E}$, are the vertices of $\E$.
We note that the bilinear form $a_{\E}(\cdot,\cdot)$
has a non-trivial kernel, given by $\P_1(\E)$. Hence,
the role of condition \eqref{numeroo} is to select
an element of the kernel of the operator.

Now, we introduce our local virtual space:
\begin{align*}\label{localspace}
\WK
:=\left\{v_h\in\VK : \int_{\E}(\PiK v_h)q=\int_{\E}v_hq\qquad\forall q\in\P_{2}(\E)\right\}.
\end{align*}

It is easy to check that $\WK\subseteq\VK$. Therefore, the operator $\PiK$
is well defined on $\WK$ and computable only on the basis of the output
values of the operators in ${\bf D_1}$ and ${\bf D_2}$.

In \cite[Lemma 2.1]{ABSVsinum16} has been established that the
set of operators ${\bf D_1}$ and ${\bf D_2}$ constitutes
a set of degrees of freedom for the space $\WK$.
Moreover, it is easy to check that
$\P_{2}\subseteq\WK$. This will guarantee the good
approximation properties for the space.

Additionaly, we have that the $\LO$ projector operator
$\PioK:\WK\to\P_{2}(\E)$ is computable from the set of degrees
freedom. In fact, for all $v_h\in\WK$, the function
$\PioK v_h\in\P_{2}(\E)$ is defined by:
\begin{equation}\label{fff}
\int_{\E}(\PioK v_h)q=\int_{\E}v_hq\qquad\forall q\in\P_{2}(\E).
\end{equation}
Now, due to the particular property appearing in definition
of the space $\WK$, the right hand
side in \eqref{fff} is computable using $\PiK v_h$,
and thus $\PioK v_h$ depends only on the values of
the degrees of freedom for $v_h$ and $\nabla v_h$.
Actually, it is easy to check that on the space
$\WK$ the projectors $\PioK v_h$ and $\PiK v_h$
are the same operator. In fact:
\begin{equation}\label{L2proj}
\int_{\E}(\PioK v_h)q=\int_{\E}v_hq=\int_{\E}(\PiK v_h)q\qquad\forall q\in\P_{2}(\E).
\end{equation}
In what follows, we keep the notation $\PiK$ for both operators.

We can now present the global virtual space: for every
decomposition $\CT_h$ of $\O$ into simple polygons $\E$, we define
$$
\Wh:=\left\{v_h\in\HdoO:\ v_h|_{\E}\in\WK\right\}.
$$
A set of degrees of freedom for $\Wh$ is given by
all pointwise values of $v_h$ on all vertices of $\CT_h$
together to all pointwise values of $\nabla v_h$ on all vertices of $\CT_h$,
excluding the vertices on $\G$ (where the values vanishes).
Thus, the dimension of $\Wh$ is tree times the number of interior vertices.

On the other hand, let $s_{K}(\cdot,\cdot)$ and $s_{K}^0(\cdot,\cdot)$
be any symmetric positive definite bilinear forms to be chosen as to satisfy:
\begin{align}
c_0a_K(v_h,v_h)\leq s_\E(v_h,v_h)\leq c_1 a_K(v_h,v_h)&\quad \forall v_h \in W_h^K\quad \mbox{with }\quad \PiK v_h =0 \label{term-stab-SK},\\[1ex]
 c_2b_K(v_h,v_h)\leq s_\E^0(v_h,v_h)\leq c_3 b_K(v_h,v_h)&\quad \forall v_h \in W_h^K. \label{term-stab-SK0}
\end{align}
%

Then, we set
\begin{align*}
a_h(u_h,v_h)
:=\sum_{\E\in\CT_h}a_{h,\E}(u_h,v_h),
\qquad u_h,v_h\in\Wh,\\
b_h(u_h,v_h)
:=\sum_{\E\in\CT_h}b_{h,\E}(u_h,v_h),
\qquad u_h,v_h\in\Wh,
\end{align*}
where $a_{h,\E}(\cdot,\cdot)$ and $b_{h,\E}(\cdot,\cdot)$
are the local bilinear forms defined on
$\WK\times\WK$ by
\begin{align}
a_{h,\E}(u_h,v_h)
:=a_{\E}\big(\PiK u_h,\PiK v_h\big)
+s_{\E}\big(u_h-\PiK u_h,v_h-\PiK v_h\big),
\qquad u_h,v_h\in\WK,\\
b_{h,\E}(u_h,v_h)\label{locforma2}
:=b_{\E}\big(\PiK u_h,\PiK v_h\big)
+s_{K}^{0}\big(u_h-\PiK u_h,v_h-\PiK v_h\big),
\qquad u_h,v_h\in\WK. 
\end{align}
%


The construction of the bilinear forms $a_{h,\E}(\cdot,\cdot)$
and $b_{h,\E}(\cdot,\cdot)$ guarantees the usual consistency
and stability properties of VEM, as noted in the Proposition below.
Since the proof follows standard arguments in the Virtual
Element literature (see \cite{ABSVsinum16,BBCMMR2013,BLRXX}) it is omitted.


\begin{proposition}
The local bilinear forms 
bilinear forms $a_{h,\E}(\cdot,\cdot)$
and $b_{h,\E}(\cdot,\cdot)$ on each element $\E$ satisfy
\begin{itemize}
\item \textit{Consistency}: for all $h > 0$ and for all $\E\in\CT_h$ we have that
\begin{align}
a_{h,\E}(p,v_h)
=a_{\E}(p,v_h)
\qquad\forall p\in\P_2(\E),
\quad\forall v_h\in\WK,\label{consis-a}\\ 
b_{h,\E}(p,v_h)
=b_{\E}(p,v_h)
\qquad\forall p\in\P_2(\E),
\quad\forall v_h\in\WK\label{consis-b}. 
\end{align}
\item \textit{Stability}: There exist positive constants
$\alpha_i, i=1,2,3,4,$ independent of $\E$, such that:
\begin{align}
\alpha_1 a_{\E}(v_h,v_h)
\leq a_{h,\E}(v_h,v_h)
\leq\alpha_2 a_{\E}(v_h,v_h)
\qquad\forall v_h\in\WK,\label{stab-a}\\
\alpha_3 b_{\E}(v_h,v_h)
\leq b_{h,\E}(v_h,v_h)
\leq\alpha_4 b_{\E}(v_h,v_h)
\qquad\forall v_h\in\WK\label{stab-b}. 
\end{align}
\end{itemize}
\end{proposition}

Now, we are in a position to write the virtual
element discretization of problem~\eqref{vibrapri}.

Find $(\l_h,w_h)\in\R\times\Wh$, $w_h\ne0$, such that
\begin{equation}\label{P11}
a_h(w_h,v_h)
=\l_h b_h(w_h,v_h)
\qquad\forall v_h\in\Wh.
\end{equation}

We observe that by virtue of \eqref{stab-a}, the
bilinear form $a_{h}(\cdot,\cdot)$ is bounded. Moreover, as shown in
the following lemma, it is also uniformly elliptic.
\begin{lemma}
\label{ha-elipt-disc}
There exists a constant $\beta>0$, independent of $h$, such that
$$
a_{h}(v_h,v_h)
\ge\beta\left\|v_h\right\|_{2,\O}^2
\qquad\forall v_h\in\Wh.
$$
\end{lemma}
\begin{proof}
Thanks to \eqref{stab-a} and Lemma~\ref{ha-elipt}, it is easy to check that
the above inequality holds with
$\beta:=\alpha\min\left\{\alpha_{1},1\right\}$.
\end{proof}

The discrete version of the operator $T$ is then given by
\begin{align*}
T_h:\ \Wh & \longrightarrow\Wh,
\\
f_h & \longmapsto T_hf_h:=u_h,
\end{align*}
where $u_h\in\Wh$ is the solution of the corresponding
discrete source problem
$$
a_h(u_h,v_h)=b_{h}(f_h,v_h)
\qquad\forall v_h\in\Wh.
$$

Because of Lemma~\ref{ha-elipt-disc}, the linear operator $T_h$ is well
defined and bounded uniformly with respect to $h$. Once more, as in the
continuous case, $(\l_h,w_h)\in\R\times\Wh$ solves problem~\eqref{P11}
if and only if $T_hw_h=\mu_h w_h$ with
$\mu_h\neq0$ and $w_h\ne0$, in which case $\mu_h:=\frac1{\l_h}$.
Moreover, $T_h$ is self-adjoint with
respect to $a_{h}(\cdot,\cdot)$.
Because of this, it is easy to prove the following
spectral characterization.

\begin{theorem}
\label{CHAR_SP_DISC}
The spectrum of $T_h$ consists of $M_h:=\dim(\Wh)$ eigenvalues,
repeated according to their respective multiplicities.
All of them are real and positive.
\end{theorem}


In order to prove that the solutions of the discrete problem~\eqref{P11} converge to those
of the continuous problem~\eqref{vibrapri}, the standard procedure would be to show that $T_h$
converges in norm to $T$ as $h$ goes to zero. However, such a proof does not seem straightforward
in our case. In fact, the operator $T_h$ is not well defined for any $f\in\HdoO$, since
the definition of bilinear form $b_{h,\E}(\cdot,\cdot)$ in \eqref{locforma2}
needs the degrees of freedom and in particular the pointwise values
of $f$, but it is for any $f\in\Wh$.

To circumvent this drawback, we will resort instead to the spectral theory
from \cite{DNR1} and \cite{DNR2}. In spite of the fact that
the main use of this theory is when $T$ is a non-compact operator,
it can also be applied to compact $T$ and we will show that in our case it works.

With this aim, we first recall the following approximation result
which is derived by interpolation
between Sobolev spaces (see for instance \cite[Theorem I.1.4]{GR}
from the analogous result for integer values of $s$).
In its turn, the result for integer values is stated
in \cite[Proposition 4.2]{BBCMMR2013} and follows from the
classical Scott-Dupont theory (see \cite{BS-2008}
and \cite[Proposition 3.1]{ABSVsinum16}):
\begin{proposition}\label{app1}
Let $v\in H^{2+s}(\E)$ with $s\in(1/2,1]$.
Assume  that assumption {\bf A2} is satisfied.
Then, there exist $v_{\pi}\in\P_{2}(\E)$
and $C>0$ such that
$$\vert v-v_{\pi}\vert_{\ell,\E}\le Ch_{\E}^{2+s-\ell}\vert v\vert_{2+s,\E}, \quad \ell=0,1,2.$$
\end{proposition}

For the analysis we will introduce the broken $H^{2}$-seminorm:
$$|v|_{2,h}^{2}:=\sum_{\E\in\CT_h}|v|_{2,\E}^{2},$$
which is well defined for every $v\in L^{2}(\O)$ such that
$v|_{\E}\in H^{2}(\E)$ for all polygon $\E\in \CT_{h}$.

Now, for $v\in\Wh$, let $\Pi_h$ be defined in $\LO$ by $(\Pi_h v)|_\E:=\PiK v$
for all $\E\in \CT_h$, where $\PiK$ has been defined in \eqref{numero}-\eqref{numeroo}.
\begin{lemma}
 \label{A}
Let $v\in\Wh$. Then, there exists $C>0$ such that 
\begin{align*}
\Vert v-\Pi_h v\Vert_{0,\O}\leq 
Ch^2\Vert v\Vert_{2,\O}.
\end{align*}
\end{lemma}
\begin{proof}
Let $v\in\Wh$. Now, let $\PiK v\in\P_{2}(\E)$ as defined in \eqref{numero}-\eqref{numeroo}.
We have for all $r\in\P_{2}(\E)$ that
$$\Vert v-\PiK v\Vert_{0,\E}^2
=\int_{\E}\left( v-\PiK v\right)\left( v-\PiK v \right)
=\int_{\E}\left( v-\PiK v\right)\left( v-r  \right).$$
Thus,
$$\Vert v-\PiK v\Vert_{0,\E}\leq\inf_{r\in\P_2(\E)} \| v-r \|_{0,K}\leq
Ch_\E^2\Vert v\Vert_{2,\E},$$
where we have used \eqref{L2proj} and \cite[Proposition 4.1]{MRR2015}
and the result follows.
\end{proof}

Now, the remainder of this section is devoted
to prove the following properties which will be
used in the sequel:

\begin{lemma}
\label{lemcotste}
There exists $C>0$ such that, for all $f_h\in\Wh$, if $u=Tf_h$ and
$u_h=T_hf_h$, then
$$
\left\|\left(T-T_h\right)f_h\right\|_{2,\O}
=\left\|u-u_h\right\|_{2,\O}
\le C\left(\left\|\Pi_h f_h-f_h\right\|_{0,\O}+\left\|u-u_I\right\|_{2,\O}
+\left|u-u_{\pi}\right|_{2,h}\right),
$$
for all $u_I\in\Wh$ and for all $u_{\pi}\in\LO$ such that
$u_{\pi}|_{\E}\in\P_2(\E) \quad \forall \E\in\CT_h$.
\end{lemma}

\begin{proof}
Let $f_h\in \Wh$. For $u_I\in\Wh$, we set $v_h:=u_h-u_I$. Thus
\begin{equation}\label{convThplusminusu_I}
|| (T-T_h)f_h||_{2,\O}\leq || u-u_I||_{2,\O} +||v_h ||_{2,\O}.
\end{equation}
Now, thanks to Lemma~\ref{ha-elipt-disc}, the definition
of $a_{h,{\E}}(\cdot,\cdot)$ and those of $T$ and $T_h$, we have
\begin{align}
\beta || v_h||_{2,\O}^2& \leq a_h(v_h,v_h)=a_h(u_h,v_h)-a_h(u_I,v_h)=b_h(f_h,v_h)-\sum\limits_{K\in \mathcal{T}_h}a_{h,K}(u_I,v_h)\nonumber\\
&  = b_h(f_h,v_h)-\sum\limits_{K\in \mathcal{T}_h} \Big\{ a_{h,K}(u_I-u_{\pi},v_h) +a_{h,K}(u_{\pi},v_h) \Big\}\nonumber\\ 
&  = b_h(f_h,v_h)-\sum\limits_{K\in \mathcal{T}_h} \Big\{ a_{h,K}(u_I-u_{\pi},v_h) +a_K(u_\pi-u,v_h)+a_{K}(u,v_h) \Big\}\nonumber\\
&  = b_h(f_h,v_h)-b(f_h,v_h)-\sum\limits_{K\in \mathcal{T}_h} \Big\{ a_{h,K}(u_I-u_{\pi},v_h) +a_K(u_\pi-u,v_h) \Big\}\label{eqconv2}.
\end{align}
Then, we bound the first term on the right hand
side of the previous inequality as follows
%
\begin{align}
b_h(f_h,v_h)-b(f_h,v_h)&=\sum\limits_{K\in \mathcal{T}_h}\Big\{ b_{h,K}(f_h,v_h)-b_K(f_h,v_h)\Big\}\nonumber\\
&=\sum\limits_{K\in \mathcal{T}_h}\Big\{ b_{h,K}(f_h-\PiK f_h,v_h)-b_K(f_h-\PiK f_h,v_h)\Big\}\nonumber\\
&\leq\sum\limits_{K\in \mathcal{T}_h}\Big\{ b_{h,K}(f_h-\PiK f_h,f_h-\PiK f_h)^{1/2}b_{h,K}(v_h,v_h)^{1/2}- ||f_h-\PiK f_h||_{0,K}||v_h ||_{0,K}\Big\}\nonumber\\
&\leq C \sum\limits_{K\in \mathcal{T}_h}||f_h-\PiK f_h||_{0,K}||v_h ||_{0,K},
\end{align}
where we have used the consistency, Cauchy-Schwartz inequality and stability
of $b_{h,K}(\cdot,\cdot)$.

Thus, from \eqref{eqconv2}, using the above bound together with the
Cauchy-Schwartz and triangular inequalities, we obtain
\begin{align*}
\beta\left\|v_h\right\|^2_{2,\O}
&\le C\sum_{\E\in\CT_h}
\left\|\PiK f_h-f_h\right\|_{0,\E}
\left\|v_h\right\|_{0,\E}+\sum_{\E\in\CT_h}
\left(\alpha_{2}\vert u_{I}-u_{\pi}\vert_{2,\E}
+\vert u_{\pi}-u\vert_{2,\E}\right)\vert v_h\vert_{2,\E}\\
&\le C\left(\sum_{\E\in\CT_h}
\left\|\PiK f_h-f_h\right\|_{0,\E}^{2}
+\vert u_{I}-u\vert_{2,\E}^{2}
+\vert u_{\pi}-u\vert_{2,\E}^{2}\right)^{1/2}\Vert v_h\Vert_{2,\O}\\
&\le C\left(
\left\|\Pi_h f_h-f_h\right\|_{0,\O}
+\left\|u_{I}-u\right\|_{2,\O}
+\vert u_{\pi}-u\vert_{2,h}\right)\Vert v_h\Vert_{2,\O}.
\end{align*}
Therefore, the proof follows from \eqref{convThplusminusu_I}
and the above inequality.
\end{proof}

The next step is to find appropriate term $u_I$ that can be used in
the above lemma. Thus, we have the following result.
\begin{proposition}\label{app2}
Assume {\textbf{A1}--\textbf{A2}} are satisfied,
let $v\in\HdsO$ with $s\in(1/2,1]$. Then, there exist $v_{I}\in\Wh$
and $C>0$ such that
$$\Vert v-v_{I}\Vert_{2,\O}\le Ch^s\vert v\vert_{2+s,\O}.$$
\end{proposition}
\begin{proof}
The proof follows repeating the arguments
from \cite[Proposition~4.4]{BMR16}
(see also \cite[Proposition~3.1]{ABSVsinum16}).
\end{proof}

As we mention before, to prove that $T_h$
provides a correct spectral approximation of 
$T$, we will resort to the theory developed
in \cite{DNR1} for noncompact operators. To this 
end, we first introduce some notations.
For any linear operator $S: 
\HdoO \longrightarrow \HdoO$, we define the norm
\begin{equation*}
\Vert S\Vert_h:=\sup_{0\ne v_h\in\Wh}\dfrac{\|S v_h\|_{2,\O}}{\|v_h\|_ {2,\O}}.
\end{equation*}

Moreover, we recall the definition of the
gap $\widehat{\delta}$ between two closed subspaces
$\CM$ and $\CN$ of $\HdoO$:
$$\widehat{\delta}(\CM,\CN):=\max\{\delta(\CM,\CN),\delta(\CN,\CM)\},$$
where
$$
\delta(\CM,\CN)
:=\sup_{x\in\CM:\ 
\left\|x\right\|_{2,\O}=1}
\left(\inf_{y\in\CN}\left\|x-y\right\|_{2,\O}\right).
$$

The theory from \cite{DNR1} guarantees approximation
of the spectrum of $T$, provide the following two properties are satisfied:
\begin{itemize}
\item (P1): $\|T-T_h\|_h\rightarrow 0,\quad \text{as } h\rightarrow 0$,
\item (P2): $\displaystyle\forall 
\phi\in \HdoO,\quad
\lim_{h\rightarrow 
0}\delta(\phi,\Wh)=0$.
\end{itemize}

Property (P2) follows immediately from the approximation
property of the virtual element space (see Proposition \ref{app2})
and the density of smooth functions in $\HdoO$.
Property (P1) is a consequence of the following lemma.

\begin{lemma}
There exist $C>0$ and $s\in(1/2,1]$, independent of h, such that
$$\|T-T_h\|_h\le Ch^s.$$
\end{lemma}
\begin{proof}
Given $f_h\in\Wh$, we have that (see Lemma~\ref{lemcotste})
$$
\left\|\left(T-T_h\right)f_h\right\|_{2,\O}
=\left\|u-u_h\right\|_{2,\O}
\le C\left(\left\|\Pi_h f_h-f_h\right\|_{0,\O}+\left\|u-u_I\right\|_{2,\O}
+\left|u-u_{\pi}\right|_{2,h}\right),
$$
now, using Lemma~\ref{A}, Propositions~\ref{app1} and \ref{app2}
and Lemma~\ref{LEM:REG}, we have
$$
\left\|\left(T-T_h\right)f_h\right\|_{2,\O}
\le
C\left(h^2\Vert f_h\Vert_{2,\O}+h^s\Vert f_h\Vert_{0,\O}\right)\le Ch^s\Vert f_h\Vert_{2,\O}.
$$
The proof is complete.
\end{proof}

\setcounter{equation}{0}
\section{Convergence and error estimates}
\label{SEC:convergence}

In this section we will adapt the arguments from \cite{DNR1,DNR2}
to prove convergence of our spectral approximation as well as to
obtain error estimates for the approximate eigenvalues and
eigenfunctions.



The following results are consequence of property (P1) (see \cite{DNR1}):
\begin{lemma}\label{bound-oper-resolv}
Suppose that (P1) holds true and let $F\subset \rho(T)$ be closed.
Then, there exist positive constants $C$ and $h_0$ independent of $h$,
such that for $h< h_0$
\begin{equation*}\label{eq-lemma-boundoperresolv}
\sup_{v_h\in \Wh} ||R_z(T_h)v_h ||_{2,\O}\leq
C ||v_h ||_{2,\O} \qquad \forall z\in F.
\end{equation*}
\end{lemma}

\begin{theorem}
Let $U\subset \mathbb{C}$ be an open set containing $\sp(T)$.
Then, there exists $h_0>0$ such that $\sp(T_h)\subset U$ for all $h<h_0$.
\end{theorem}

An immediate consequence of this theorem is that the proposed
virtual element method does not introduce spurious
modes with eigenvalues interspersed among those with a physical meaning.

By applying the results from \cite{DNR1} to our problem,
we conclude the spectral convergence of $T_h$ to $T$  as $h\to0$.
More precisely, let $\mu\ne0$ be an isolated eigenvalue
of $T$ with multiplicity $m$ and let $\mathcal{C}$
be an open circle in the complex plane centered at $\mu$,
such that $\mu$ is the only eigenvalue of $T$ lying in $\mathcal{C}$
and $\partial \mathcal{C}\cap\sp(T)=\emptyset$. Then, according
to Section~2 in \cite{DNR1} for $h$ small enough there exist
$m$ eigenvalues $\mu_h^{(1)},\ldots,\mu_h^{(m)}$ of $T_h$
(repeated according to their respective multiplicities)
which lie in $\mathcal{C}$. Therefore, these eigenvalues
$\mu_h^{(1)},\ldots,\mu_h^{(m)}$ converge to $\mu$ as $h$
goes to zero.

The next step is to obtain error estimates for the spectral
approximation. With this aim, we will use the theory from \cite{DNR2}.
However, we cannot apply the results from this reference directly
to our problem, because of the variational crimes in the bilinear
forms used to define the operator $T_h$. Therefore, we need to
extend the results from this reference to our case. With this purpose,
we follow an approach recently presented in \cite{BMRR}.
 
Consider the eigenspace $\mathcal{E}$ of $T$ corresponding
to $\mu$ and the $T_h-$invariant subspace $\mathcal{E}_h$
spanned by the eigenspaces of $T_h$ corresponding to
$\mu_h^{(1)},\ldots,\mu_h^{(m)}$. As a consequence of
Lemma~\ref{bound-oper-resolv}, we have that
\begin{equation*}\label{conseq-of-Lemma}
|| (zI-T_h)v_h||_{2,\O}\geq C || v_h||_{2,\O}
\qquad \forall v_h\in \Wh, \quad \forall z\in \partial \mathcal{C}, 
\end{equation*}
for $h$ small enough.

Let $\mathcal{P}_h:H_0^2(\O)\to \Wh \hookrightarrow \Hcd$
be the projector with range $\Wh$ defined by the relation 
\begin{equation*}
a(\mathcal{P}_hu-u,v_h)=0\qquad \forall v_h\in \Wh.
\end{equation*}
Notice that $\mathcal{P}_h$ is bounded uniformly on $h$
(namely $||\mathcal{P}_h u||_{2,\O}\leq ||u ||_{2,\O}$)
and
\begin{equation*}
|| u-\mathcal{P}_h u||_{2,\O}=\delta (u,W_h)\qquad \forall u\in \Hcd.
\end{equation*}

Let us define
$$\widehat{T}_h:=T_h\mathcal{P}_h:\Hcd \to \Wh.$$
Notice that $\sp(\widehat{T}_h)=\sp(T_h)\cup \{0 \}$.

Next, we introduce the following spectral projectors
(the second one, is well defined at least for $h$ small enough): 
\begin{itemize}
\item The spectral projector of $T$ relative to
$\mu$: $F:=\frac{1}{2\pi i}\int_{\partial \mathcal{C}}R_z(T)dz$;
\item The spectral projector of $\widehat{T}_h$ relative to
$\mu_h^{(1)},\ldots,\mu_h^{(m)}$: $\widehat{F}_h:=\frac{1}{2\pi i}
\int_{\partial \mathcal{C}}R_z(\widehat{T}_h)dz$.
\end{itemize}

We have the following result (see Lemma 1 in \cite{DNR2}).
\begin{lemma}\label{unif-bound-R_z}
There exist $h_0>0$ and $C>0$ such that
\begin{equation*}
||R_z(\widehat{T}_h) ||\leq C\qquad \forall z\in \partial \mathcal{C}, \quad \forall h\leq h_0.
\end{equation*}
\end{lemma} 
\begin{proof}
It is identical to that of Lemma 11 from \cite{BMRR}.
\end{proof}

Consequently, for $h$ small enough, the spectral projector $\widehat{F}_h$
is bounded uniformly in $h$.

Now, we define
$$\gamma_h:=\delta(\mathcal{E},\Wh)\quad \textrm{and}\quad
\eta_h:=\sup\limits_{w\in\mathcal{E}}\displaystyle\frac{||w-\Pi_h w ||_{0,\O}}{|| w||_{2,\O}}.$$
From Lemmas~\ref{LEM:REG} and \ref{A} we have that 
\begin{equation}\label{dist-E-to-W_h}
\gamma_h\leq Ch^{\tilde{s}}\qquad \mbox{and }\quad \eta_h\leq Ch^2,
\end{equation}
where ${\tilde{s}}\in(1/2,1]$ is such that $\mathcal{E}\subset H^{2+\tilde{s}}(\Omega)$.

The following result establish an error estimate for the eigenfunctions.
\begin{theorem}\label{conv-eigenfunc}
If $\mathcal{E}\subset H^{2+\tilde{s}}(\O)$ with $\tilde{s}>(1/2,1]$,
there exist positive constants $h_0$ and $C$ such that, for all $h<h_0$,
$$\widehat{\delta}(\mathcal{E},\mathcal{E}_h)\leq Ch^{\tilde{s}}.$$
\end{theorem}
\begin{proof}
It follows by arguing exactly as in the proof of Theorem 1 from \cite{DNR2}
and using \eqref{dist-E-to-W_h}.
\end{proof}

Finally, we have the following result that
provides an error estimate for the eigenvalues.

\begin{theorem}
 There exist positive constants $C$ and
 $h_0$ independent of $h$, such that,
 for all $h<h_0$,
\begin{equation*}
 \left|\l-\l_h^{(i)}\right|
 \le Ch^{2\tilde{s}},\qquad i=1,\ldots,m,
\end{equation*}
where ${\tilde{s}}\in(1/2,1]$ is such that $\mathcal{E}\subset H^{2+\tilde{s}}(\Omega)$.
\end{theorem}

\begin{proof}
Let $w_h$ be such that $(\l_h^{(i)},w_h)$ is a solution of
\eqref{P11} with $\left\| w_h\right\|_{2,\O}=1$. According to
Theorem~\ref{conv-eigenfunc}, $\delta(w_{h},\boldsymbol{\CE})\leq Ch^{\tilde{s}}$.
It follows that there exists $(\l,w)$ eigenpair solution of \eqref{vibrapri}
such that 
\begin{equation}\label{fin4}
\left\|w-w_h\right\|_{2,\O}\le Ch^{\tilde{s}}.
\end{equation}
From the symmetry of the bilinear forms and the facts that 
$a(w,v)=\l b(w,v)$ for all $v\in\Hcd$ (cf. \eqref{vibrapri}) and 
$a_{h}( w_h, v_h)=\l_h^{(i)}b_h(w_h,v_h)$ for all $v_h\in \Wh$ (cf.
\eqref{P11}), we have
\begin{align*}
a(w-w_{h},w-w_{h})&-\l b(w-w_{h},w-w_{h})
=a(w_{h},w_{h})-\l b(w_{h},w_{h})\\
& =a(w_{h},w_{h})-a_{h}(w_{h},w_{h})+\l_{h}^{(i)}b_{h}(w_{h},w_{h})-\l b(w_{h},w_{h})\\
& =a(w_{h},w_{h})-a_{h}(w_{h},w_{h})+\l_{h}^{(i)}\left[b_{h}(w_{h},w_{h})-b(w_{h},w_{h})\right]+(\l_{h}^{(i)}-\l )b(w_{h},w_{h})
\end{align*}
from which we obtain the following identity:
\begin{align}
\nonumber
(\l_{h}^{(i)}-\l )b(w_{h},w_{h})&=a(w-w_{h},w-w_{h})-\l b(w-w_{h},w-w_{h})\\\label{45}
&\quad+(a_{h}(w_{h},w_{h})-a(w_{h},w_{h}))-\l_{h}^{(i)}\left[b_{h}(w_{h},w_{h})-b(w_{h},w_{h})\right].
\end{align}
The next step is to bound each term on the right hand side above. The
first and the second ones are easily bounded using the Cauchy-Schwarz inequality 
and \eqref{fin4}:
\begin{equation}
\label{terms_12}
\left|a(w-w_{h},w-w_{h})-\l b(w-w_{h},w-w_{h})\right|
\le  C \left(|w-w_{h}|_{2,\O}^2+\|w-w_{h}\|_{0,\O}^2\right)\le Ch^{2\tilde{s}}.
\end{equation}

For the third term, for $w\in\boldsymbol{\CE}$, we consider $w_{\pi}\in\LO$
defined on each $\E\in\CT_h$ so that $w_{\pi}|_{\E}\in\P_2(\E)$
and the estimate of Proposition~\ref{app1} holds true.
Then, we use \eqref{consis-a} and \eqref{stab-a}  to write
\begin{align*}
|a_h(w_h,w_h)-a(w_h,w_h)|
&=\Big|\sum\limits_{K\in \mathcal{T}_h}\Big\{a_{h,K}(w_h-w_{\pi},w_h) -a_{K}(w_h-w_{\pi},w_h)\Big\}\Big|\nonumber\\
&\leq  \sum\limits_{K\in \mathcal{T}_h} (1+\alpha_2)a_K(w_h-w_{\pi},w_h-w_{\pi})\nonumber\\
& \leq C \sum\limits_{K\in \mathcal{T}_h}|w_h-w_{\pi}|_{2,K}^2. 
\end{align*}
Then, adding and subtracting $w$, using triangular inequality, Proposition~\ref{app1},
and \eqref{fin4}, we obtain
\begin{equation}
 \label{terms_{12}}
\left|a_h(w_h,w_h)-a(w_h,w_h)\right|
\le Ch^{2\tilde{s}}.
 \end{equation}

For the last term in \eqref{45}, using that $\PiK$ is also the $L^2$-projector (see \eqref{L2proj}), we obtain
\begin{align*}
 \left|b_h(w_h,w_h)-b(w_h,w_h)\right|&\leq C\sum_{\E\in\CT_h}\|w_h-\PiK  w_h\|_{0,\E}^{2}\\
 &\leq C\sum_{\E\in\CT_h}\|w_h-  w_\pi\|_{0,\E}^{2}\\
 &\leq C\sum_{\E\in\CT_h}(\|w-w_\pi\|_{0,\E}^{2}+\|w-  w_{h}\|_{0,\E}^{2})\le Ch^{2\tilde{s}},
\end{align*}
where we have used again Proposition~\ref{app1} and \eqref{fin4}.

On the other hand,
\begin{align*}
1=\|w_h\|_{2,\O}^2\leq\l_h^{(i)}b_h(w_h,w_h)\leq \l_h^{(i)}C\|w_h\|_{0,\O}^2, 
\end{align*}
thus, the theorem follows from
from \eqref{45}--\eqref{terms_{12}} and the inequalities above.
%
\end{proof}




\setcounter{equation}{0}
\section{Numerical results}
\label{SEC:NUMER}

We report in this section a couple of tests which have allowed us to
assess the theoretical results proved above. With this aim, we have
implemented in a MATLAB code the proposed VEM on arbitrary
polygonal meshes, by following the ideas presented in \cite{BBMR2014}.

Now, to complete the choice of the VEM, we had to fix the bilinear
forms $s_{\E}(\cdot,\cdot)$ and $s_\E^0(\cdot,\cdot)$ satisfying
\eqref{term-stab-SK} and \eqref{term-stab-SK0}, respectively.
Proceeding as in \cite{BBCMMR2013}, a natural choice
for $s_{\E}(\cdot,\cdot)$ is given by
\begin{align}
s_\E(u_h,v_h):=\sigma_{\E}h_K^{-2}\sum\limits_{i=1}^{N_\E}[u_h(P_i)v_h(P_i)
+h_{P_i}^2\nabla u_h(P_i)\cdot\nabla v_h(P_i)] & \quad \forall u_h,v_h\in W_h^{K},\nonumber
\end{align}
where $P_1,\ldots,P_{N_{\E}}$ are the vertices of $\E$,
$h_{P_i}$ corresponds to the maximun diameter of the elements with $P_i$ as a vertex
and $\sigma_\E>0$ is a multiplicative factor to take into account the magnitude
of the material parameter, for instance, in the numerical tests a possible
choice could be to set $\sigma_\E>0$ as the mean value of
the eigenvalues of the local matrix $a_{\E}\big(\PiK u_h,\PiK v_h\big)$.
This ensure that the stabilizing term scales as $a_{\E}(v_h,v_h)$.
Now, a choice for $s_\E^0(\cdot,\cdot)$ is given by

\begin{align*}
s_{\E}^0(u_h,v_h):=\sigma_{\E}^{0}h_K^2\sum\limits_{i=1}^{N_K}[u_h(P_i)v_h(P_i)
+h_{P_i}^2\nabla u_h(P_i)\cdot \nabla v_h(P_i)] & \quad \forall u_h,v_h\in W_h^{K}.\label{chose-sK0}
\end{align*}
In this case, we have multiplied the stabilizing
term by the parameter $\sigma_{\E}^{0}>0$, where in this case,
a possible choice could be to set $\sigma_\E^{0}>0$ as the mean value of
the eigenvalues of the local matrix $b_{\E}\big(\PiK u_h,\PiK v_h\big)$
to ensure \eqref{term-stab-SK0}. A proof of \eqref{term-stab-SK} and \eqref{term-stab-SK0}
for the above (standard) choices could be derived following the arguments in \cite{BLRXX}
(see also \cite{ABSVsinum16}).
Finally, we mention that the above definitions of the bilinear
forms $s_{\E}(\cdot,\cdot)$ and $s_\E^0(\cdot,\cdot)$
are according with the analysis presented
in \cite{MRR2015} in order to avoid spectral pollution.



We have tested the method by using
different families of meshes
(see Figure~\ref{FIG:VM1}):
\begin{itemize}
\item $\CT_h^1$: rectangular meshes;   
\item $\CT_h^2$: hexagonal meshes;
\item $\CT_h^3$: non-structured hexagonal meshes made of convex hexagons;
\item $\CT_h^4$: trapezoidal meshes which consist
of partitions of the domain into $N\times N$ congruent
trapezoids, all similar to the trapezoid with
vertices $(0,0)$, $(\dfrac{1}{2},0)$, $(\dfrac{1}{2},\dfrac{2}{3})$ and $(0,\dfrac{1}{3})$. 
\end{itemize}

The refinement parameter $N$ used to label each mesh is the number of elements
on each edge of the plate.


\begin{figure}[H]
\begin{center}
\begin{minipage}{6.3cm}
\centering\includegraphics[height=6.3cm, width=6.3cm]{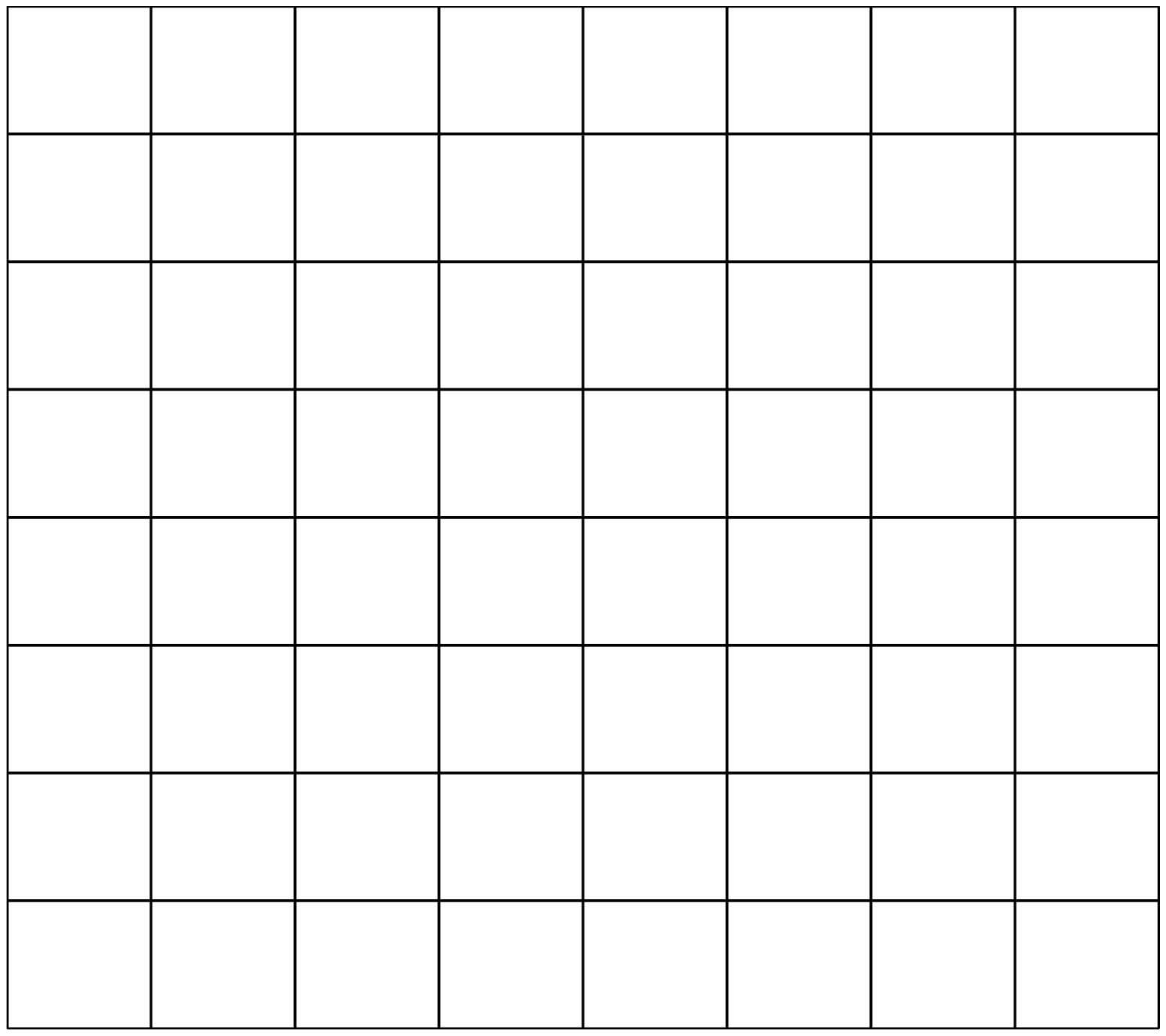}
\end{minipage}
\begin{minipage}{6.3cm}
\centering\includegraphics[height=6.3cm, width=6.3cm]{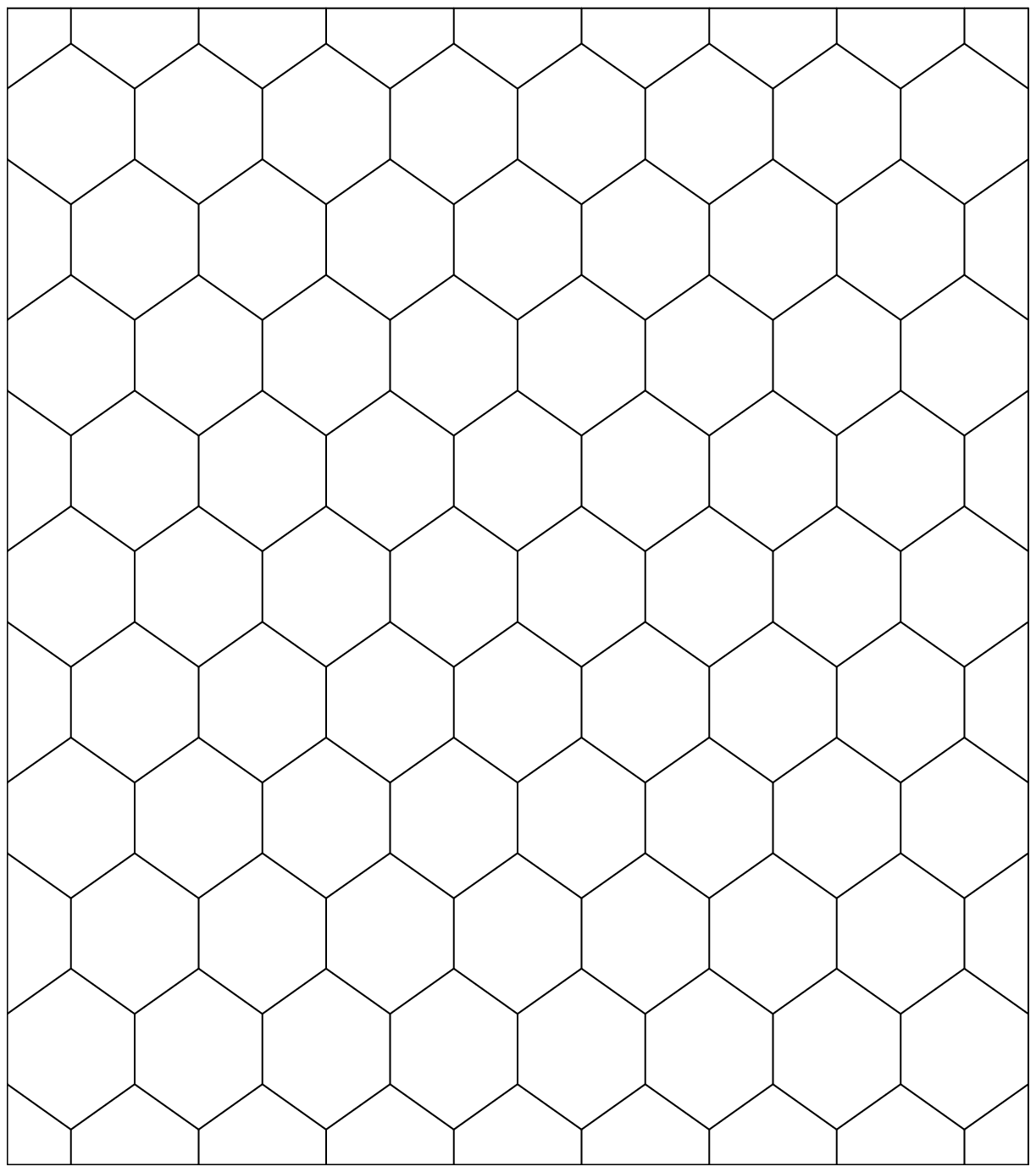}
\end{minipage}
\begin{minipage}{6.3cm}
\centering\includegraphics[height=6.3cm, width=6.3cm]{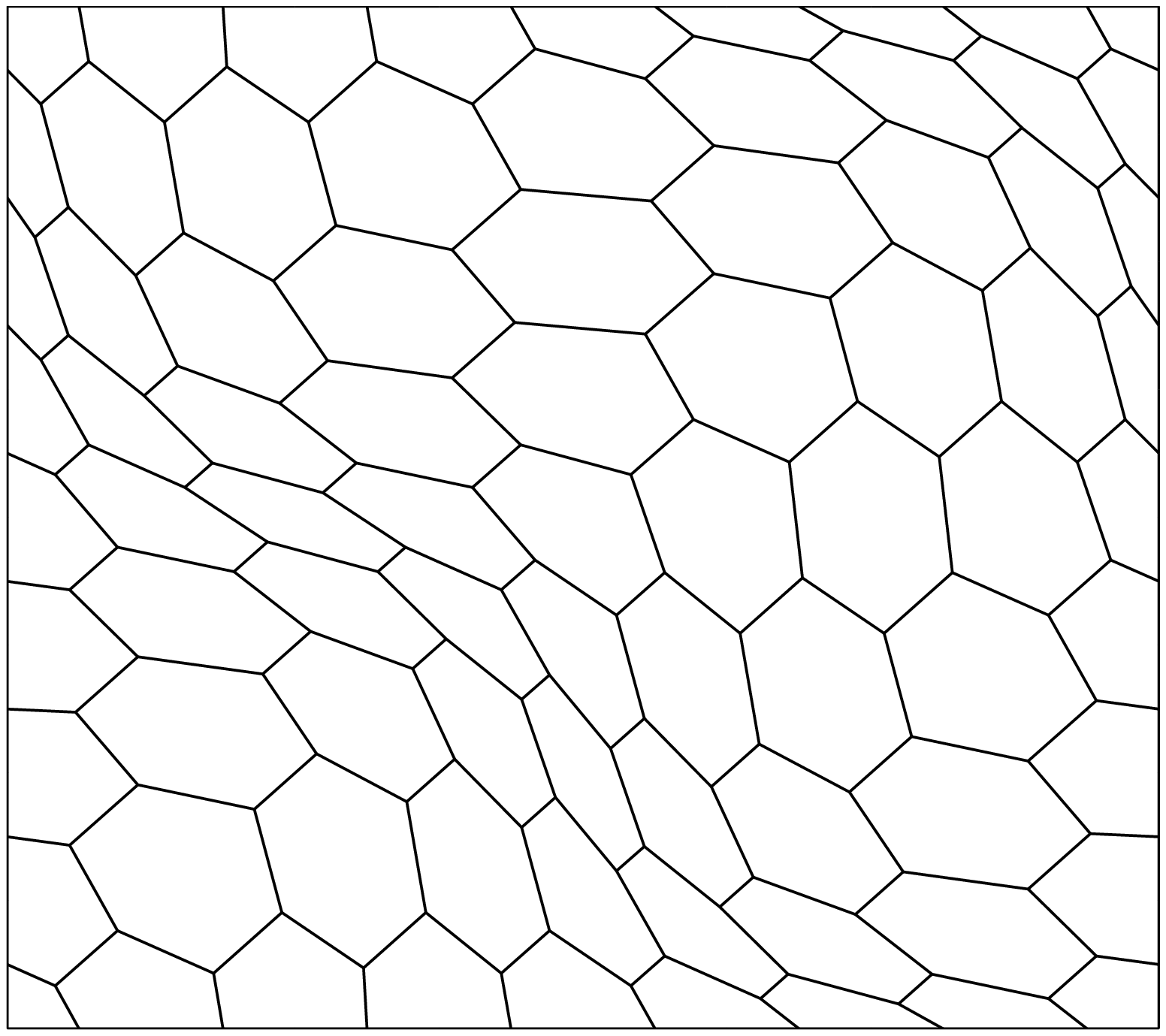}
\end{minipage}
\begin{minipage}{6.3cm}
\centering\includegraphics[height=6.3cm, width=6.3cm]{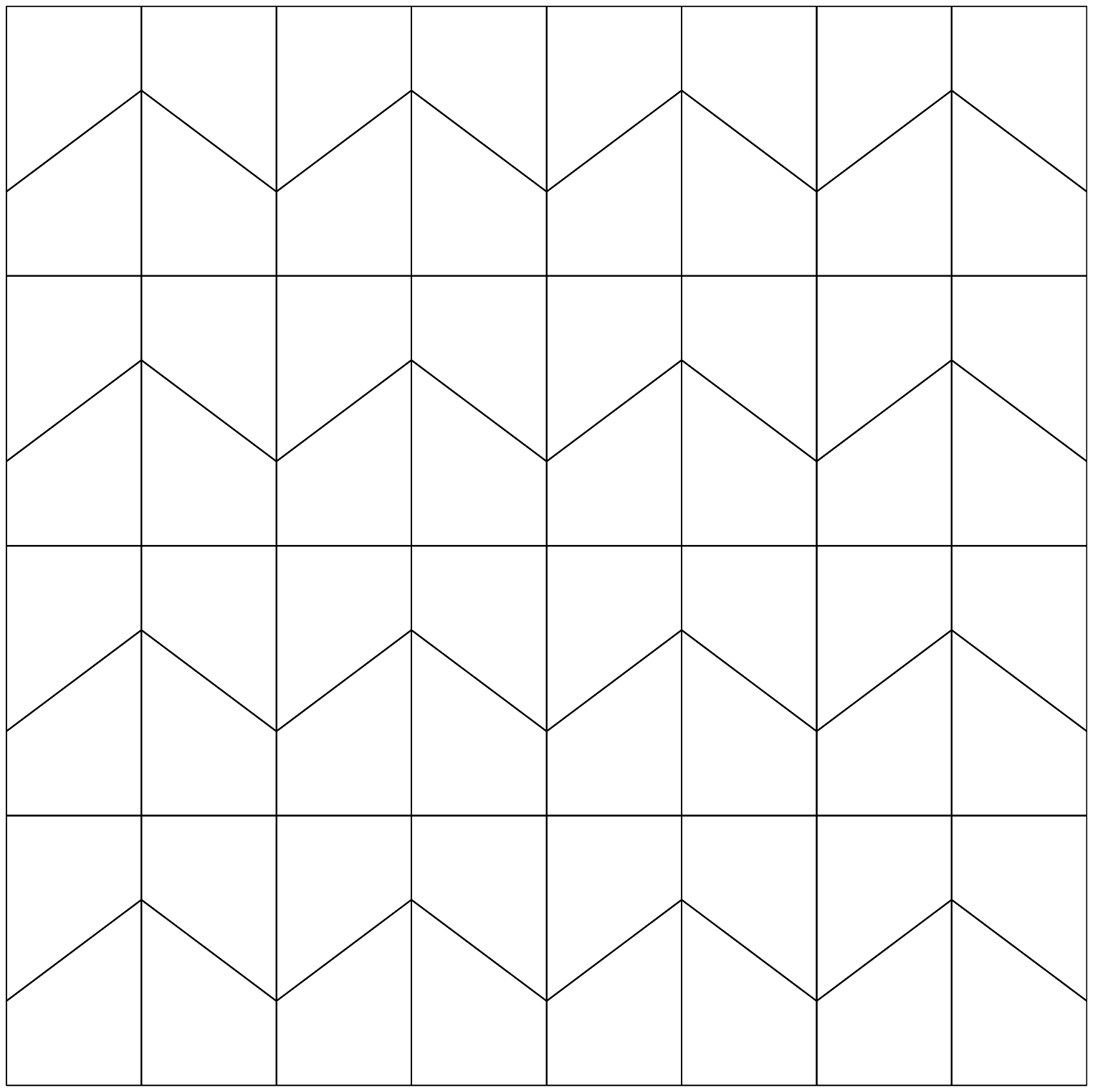}
\end{minipage}
\caption{ Sample meshes: $\CT_h^1$ (top left), $\CT_h^2$ (top right),
$\CT_h^3$ (bottom left) and $\CT_h^{4}$ (bottom right), for $N=8$.}
 \label{FIG:VM1}
\end{center}
\end{figure}

\subsection{Simply supported plate}

First, we have considered a simply supported plate,
because analytical solutions are available in this
case (see \cite{ALR05,BO}). Even though our theoretical
analysis has been developed only for clamped plates, we think that
the results of the previous sections should hold true
for more general boundary conditions as well. The results that
follow give some numerical evidence of this.
For the computations we took $\Omega:=(0, 1)^2$.

In Table~\ref{TAB:1} we report the four lowest eigenvalues
($\lambda_i,$ $i = 1, 2, 3, 4$) computed
by our method with two different family of meshes
and $N = 32, 64, 128$ for a simply supported
plate. The table includes computed orders of convergence,
as well as more accurate values extrapolated by means of a least-squares fitting.
The last column shows the exact eigenvalues.
It can be seen from Table~\ref{TAB:1} that the method converges
to the exact values with an optimal quadratic order.
Notice that, for the $\CT_h^1$ meshes,
the second computed eigenvalue is double, because the meshes
preserve the symmetry of the domain leading to an eigenvalue
of multiplicity 2 in the continuous problem.


\begin{table}[ht]
\caption{Lowest vibration frequencies of a simply supported
square plate computed on different meshes with the method analyzed in
this paper.}
\label{TAB:1}
\begin{center}
{\small\begin{tabular}{cccccccc}
\hline
 & Mesh & $N=32$ & $N=64$ & $N=128$  & Order & Extrapolated &Exact\\
\hline
  $\l_1$     &   & 390.0184 & 389.7307 & 389.6599 & 2.02 & 389.6366& 389.6364\\
  $\l_2$ & $\CT_h^1$&2430.2171 &2433.9024 &2434.8914 & 1.90 &2435.2523& 2435.2273\\
  $\l_3$     &   &2430.2171 &2433.9024 &2434.8914  & 1.90 &2435.2523 &2435.2273\\
  $\l_4$     &   & 6259.8318 & 6240.2949 & 6235.6906  & 2.09 & 6234.2872& 6234.1818\\
\hline
  $\l_1$     &   & 389.0957  & 389.4908 & 389.5987 & 1.87 & 389.6395& 389.63634\\
  $\l_2$     & $\CT_h^2$  &2412.1885 &2429.0389 &2433.6393 & 1.87 &2435.3783& 2435.2273\\
  $\l_3$     &   &2433.8095 &2434.8277 &2435.1197  & 1.80 &2435.2376 &2435.2273\\
  $\l_4$     &   & 6199.2905 & 6224.8431 & 6231.7684 & 1.88 & 6234.3634& 6234.1818\\
\hline
\end{tabular}}
\end{center}
\end{table}

\subsection{Clamped plate}

In this numerical test we took $\Omega:=(0, 1)^2$
and considered clamped boundary condition on the whole of $\partial\Omega$.
We present numerical experiments which confirm the theoretical
results proved above. 

Table~\ref{TAB:2} shows the four lowest vibration frequencies
computed with successively refined meshes of each type for
a clamped plate. The table includes orders of convergence,
as well as accurate values extrapolated by means
of a least-squares fitting.
Moreover, we compare the performance of the proposed
method with the one presented in \cite{MoRo2009}
with a mixed formulation for solving the plate vibration problem
and a Galerkin method based on
piecewise linear and continuous finite elements.
With this aim, we include in the last column of
Table~\ref{TAB:2} the values obtained by extrapolating
those computed with method in \cite{MoRo2009}
on uniform triangular meshes
as those shown in Figure~\ref{FIG:L_MESH},
for the same problem.

\begin{figure}[ht]
\begin{center}
\begin{minipage}{6.3cm} 
\centering\includegraphics[height=6.3cm, angle=-90]{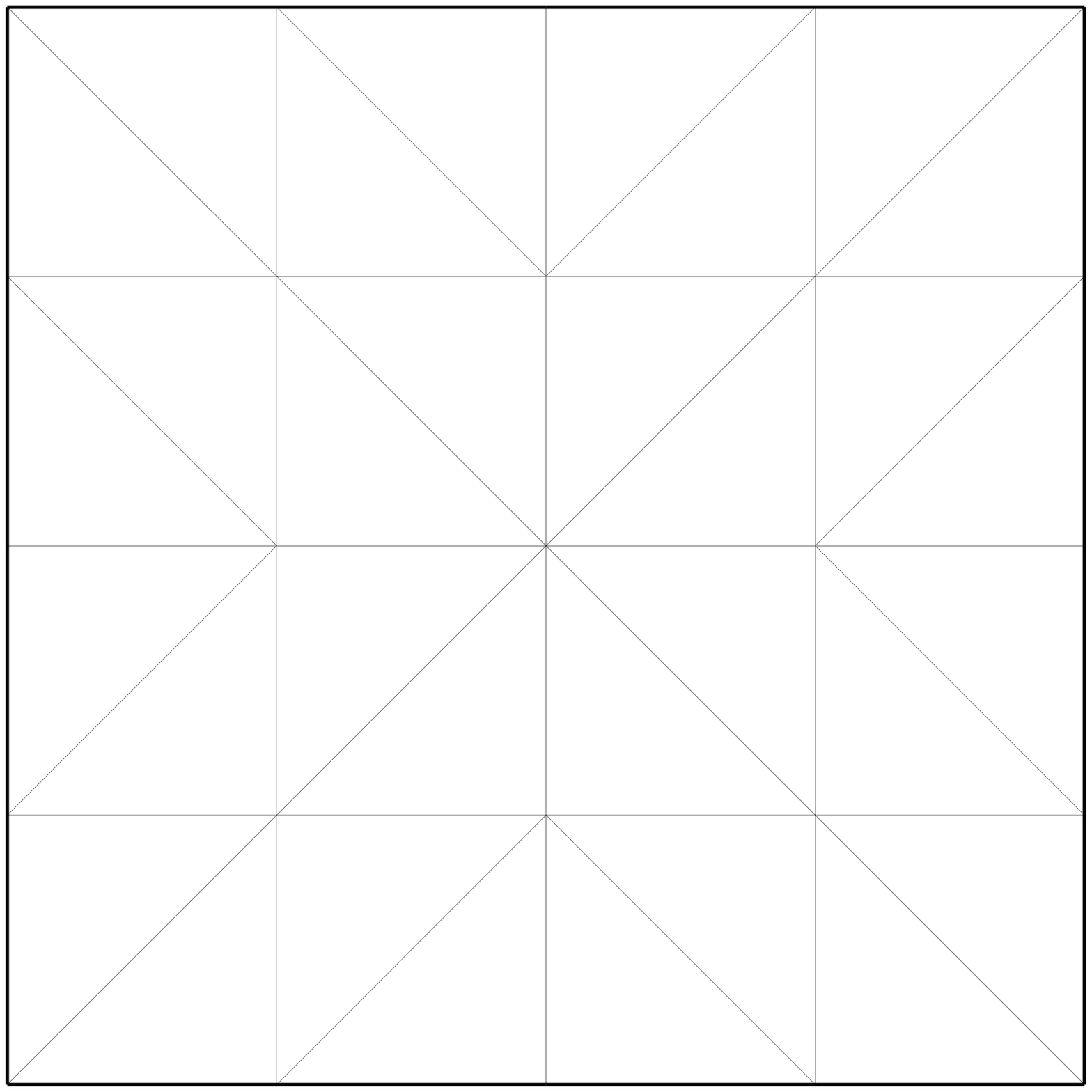}
\centering{$N=4$} 
\end{minipage}
\begin{minipage}{6.3cm}
\centering\includegraphics[height=6.3cm, angle=-90]{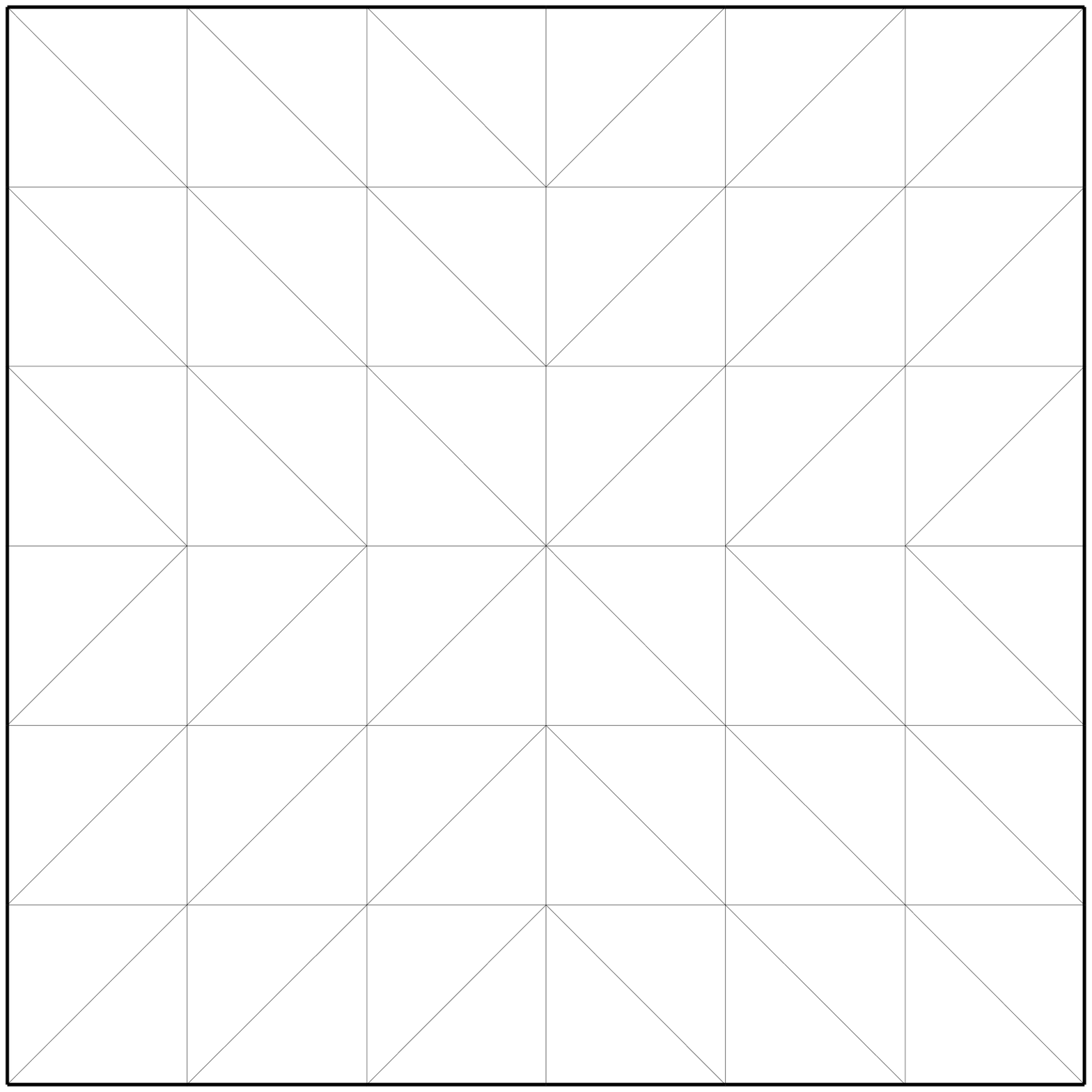}
\centering{$N=6$}
\end{minipage}
\caption{Uniform meshes.}
\label{FIG:L_MESH}
\end{center}
\end{figure}


\begin{table}[ht]
\caption{Lowest vibration frequencies of a clamped
square plate computed on different meshes with the VEM method analyzed in
this paper and the one in \cite{MoRo2009}.}
\label{TAB:2}
\begin{center}
{\small\begin{tabular}{cccccccc}
\hline
 & Mesh & $N=32$ & $N=64$ & $N=128$ & Order & Extrapolated& \cite{MoRo2009}\\
\hline
 $\l_1$       &   & 1283.2286 & 1291.5607 & 1294.0225 & 1.76 & 1295.0526&1294.9369 \\
 $\l_2$       & $\CT_h^3$ &5268.2854 &5353.1383 &5377.7322 &1.79 & 5387.6973&5386.6675 \\
 $\l_3$      &   &5326.6504 &5368.5269 &5381.6191 &1.68 & 5387.5416& 5386.6675\\
 $\l_4$      &   &11406.3068 &11622.9583 &11686.8981 &1.76 & 11713.7035&11710.9076\\
\hline
 $\l_1$       &   & 1289.7221 & 1293.6088 & 1294.6010 & 1.97 & 1294.9410& 1294.9369\\
 $\l_2$       & $\CT_h^4$ &5318.4039 &5368.9773 &5382.1939 &1.94 & 5386.8279& 5386.6675\\
  $\l_3$      &   &5351.6510 &5377.6743 &5384.3950 &1.95 & 5386.7517 &5386.6675\\
  $\l_4$      &   &11664.9586 &11698.2652 &11707.5973 &1.84 & 11711.1942& 11710.9076\\
\hline
\end{tabular}}
\end{center}
\end{table}

It is clear that the eigenvalue approximation order of
our method is quadratic and that the results
obtained by the two methods agree perfectly well.


\subsection{L-shaped plate}

Finally, we present two numerical experiments
which confirm the theoretical results proved above.
We have computed the vibration frequencies
of an L-shaped plate: $\O:=(0,1)\times(0,1)\setminus[0.5,1)\times[0.5,1)$.

In the first test we considered clamped boundary condition
on the whole of $\partial\Omega$ and
we have used uniform triangular meshes as
those shown in Figure~\ref{FIG:VM37}.
Once again, we compare the performance of the proposed
method with the one presented in \cite{MoRo2009}.

We report in Table~\ref{TAB:4} the four lowest vibration
frequencies computed with the method analyzed in this paper.
The table includes orders of convergence,
as well as accurate values extrapolated by means
of a least-squares fitting.
The last column shows the values obtained by extrapolating
those computed with method in \cite{MoRo2009}
on the same uniform triangular meshes.


In this case, for the first vibration frequency, the method converges
with order close to $1.28$, which is the expected one because of the
singularity of the solution (see \cite{G}). Instead, the method
converges with larger orders for the second, third and fourth
vibration frequencies.

\begin{table}[ht]
\caption{Lowest vibration frequencies of an L-shaped clamped plate
computed on uniform triangular meshes with the VEM method analyzed in this paper
and the one in \cite{MoRo2009}.}
\label{TAB:4}
\begin{center}
{\small\begin{tabular}{rrrrrrr}
\hline
 & $N=32$ & $N=64$ & $N=128$ & Order & Extrapolated &\cite{MoRo2009}\\
\hline
$\l_1$&   6827.5421&    6753.6207&    6725.1315 & 1.28& 6707.4264&6704.2982\\
$\l_2$&  11128.5787&   11073.4576&  11059.3867&  1.97& 11054.5647&11055.5189\\
$\l_3$&  14989.9367&  14926.5156&   14910.6489&     2.00&  14905.3676&14907.0816\\
$\l_4$&  26325.7078&  26195.9206&   26163.4597&     2.00& 26152.6488&26157.9673\\
\hline
\end{tabular}}
\end{center}
\end{table}

In this case, we mention the following advantages of the proposed VEM method:
the computational cost of our method
is smaller than the method studied in \cite{MoRo2009}.
In fact, the number of unknowns for our VEM method is,
3$N_v$, where $N_v$ denotes the number of vertices,
whereas in \cite{MoRo2009} is 4$N_v$.
Moreover, in this case, the eigenvalue problem to be solved
is much simpler than the one arising from the formulation studied
in \cite{MoRo2009}. In fact, the latter leads to a degenerate
generalized matrix eigenvalue
problem, which is shown to be well posed in \cite[Appendix]{MoRo2009}
but that cannot be solved with standard eigensolvers.

We show in Figure~\ref{FIG:VM37}
the eigenfunctions corresponding to the four lowest eigenvalues
for an L-shaped clamped plate.

\begin{figure}[H]
\begin{center}
\begin{minipage}{6.3cm}
\centering\includegraphics[height=6cm, width=6.9cm]{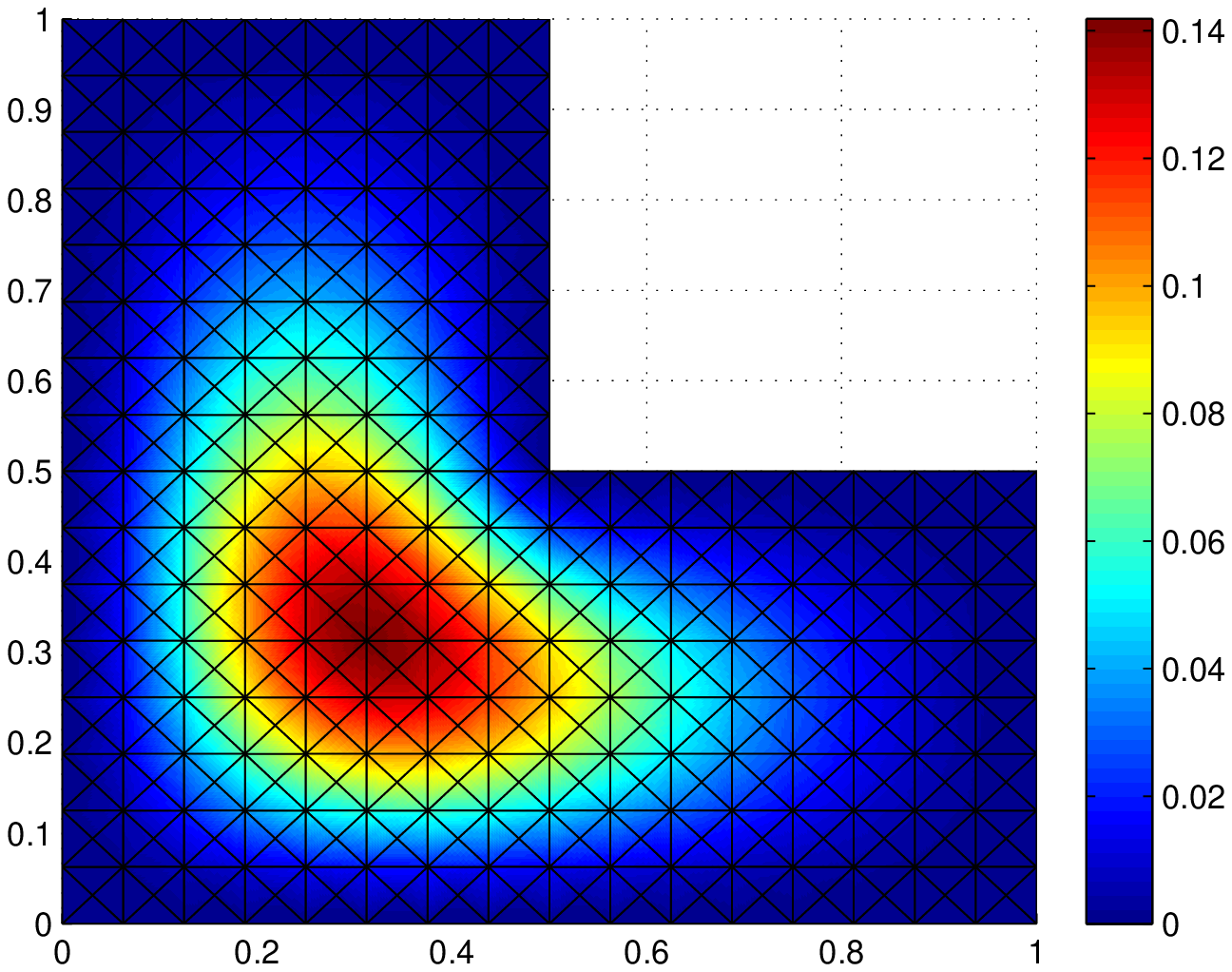}
\end{minipage}
\begin{minipage}{6.3cm}
\centering\includegraphics[height=6cm, width=6.9cm]{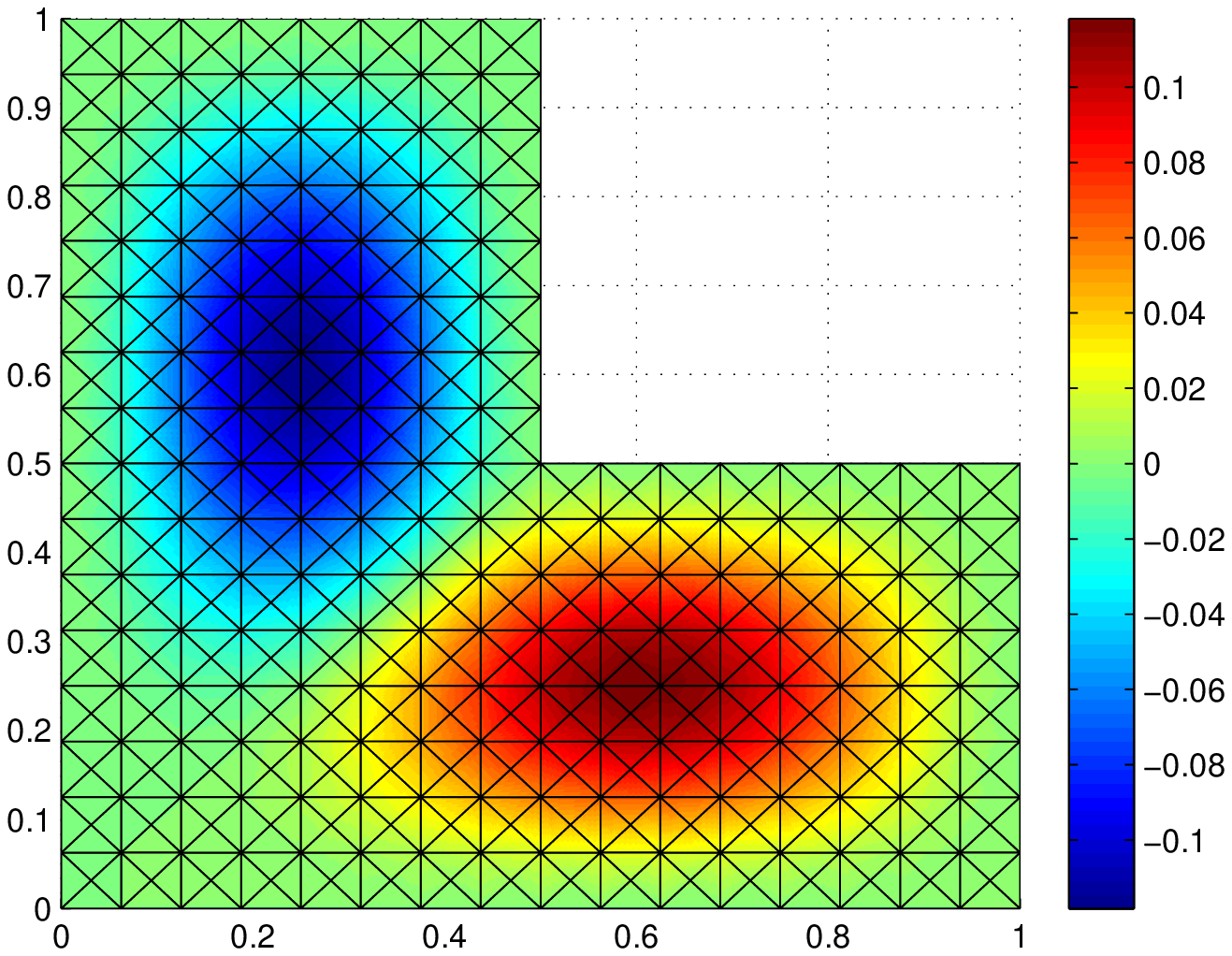}
\end{minipage}
\begin{minipage}{6.3cm}
\centering\includegraphics[height=6cm, width=6.9cm]{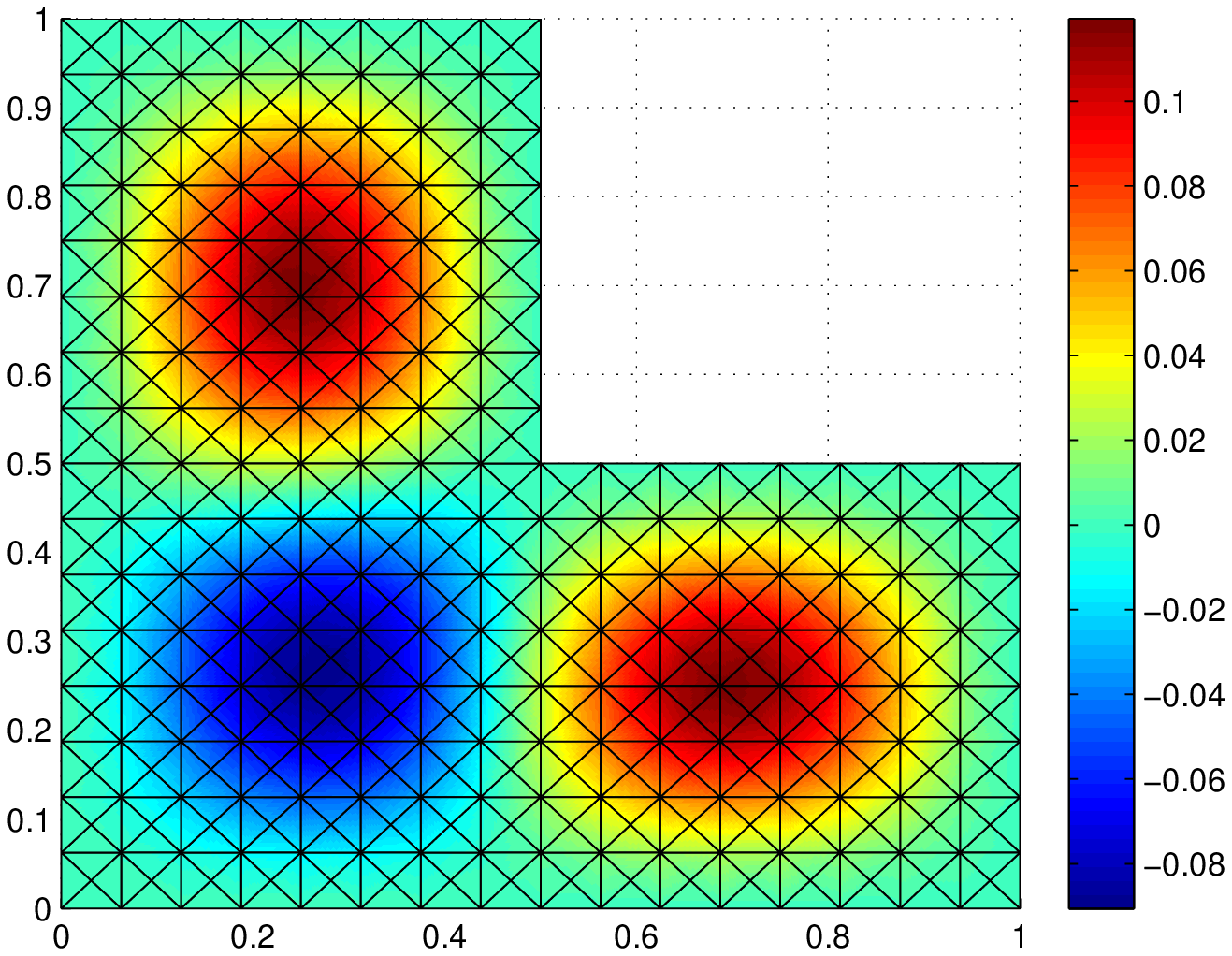}
\end{minipage}
\begin{minipage}{6.3cm}
\centering\includegraphics[height=6cm, width=6.9cm]{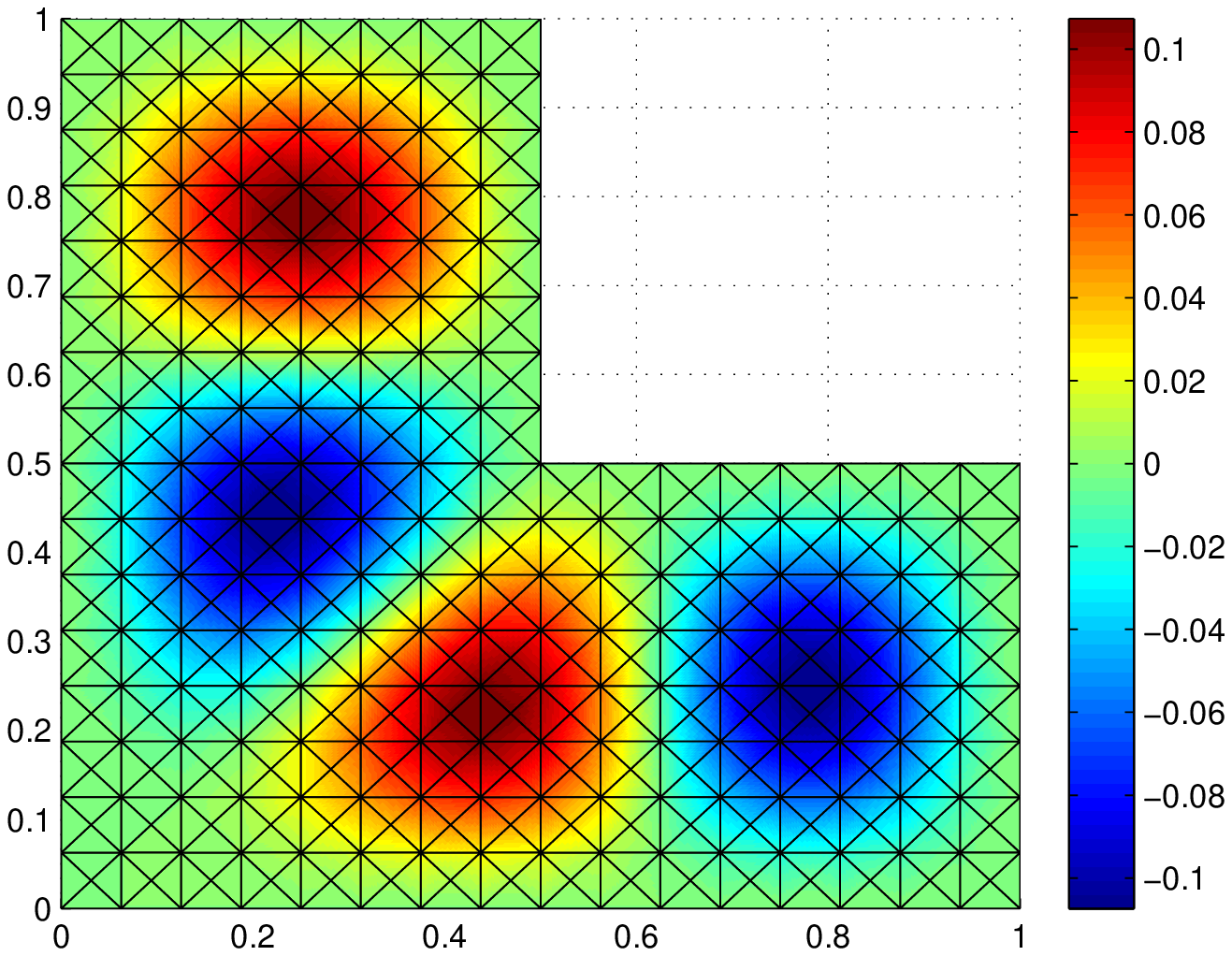}
\end{minipage}
\caption{Eigenfunctions of the plate problem with clamped boundary
condition associated with eigenvalues $\lambda_1$ (top left)
$\lambda_2$ (top right), $\lambda_3$ (bottom left)
and  $\lambda_4$ (bottom right).}
\label{FIG:VM37}
\end{center}
\end{figure}


Finally, Table~\ref{TAB:5} shows the four lowest vibration
frequencies computed with successively refined triangular meshes
for an L-shaped clamped-free plate.
The table includes orders of convergence, as well as accurate
values extrapolated by means of a least-squares fitting.
We observe from the results reported in Table~\ref{TAB:5}
that the order of convergence is
again quadratic in this case.


\begin{table}[ht]
\caption{Lowest vibration frequencies of an L-shaped clamped-free plate
computed on triangular meshes with the VEM method analyzed in this paper.}
\label{TAB:5}
\begin{center}
{\small\begin{tabular}{rrrrrrr}
\hline
 & $N=32$ & $N=64$ & $N=128$ & Order & Extrapolated \\
\hline
$\l_1$&   1198.2579&    1195.3003&    1194.4606 & 1.82& 1194.1302\\
$\l_2$&  4576.4950&   4556.9217&  4551.0233&  1.73& 4548.4764\\
$\l_3$&  6807.0921&  6785.8226&   6780.4745&     2.00&  6778.6710\\
$\l_4$&  15094.2896&  15019.3352&   14998.6457&     1.86& 14990.8077\\
\hline
\end{tabular}}
\end{center}
\end{table}



Finally, we show in Figure~\ref{FIG:VM36}
the eigenfunctions corresponding to the four lowest eigenvalues.


\begin{figure}[H]
\begin{center}
\begin{minipage}{6.3cm}
\centering\includegraphics[height=6cm, width=6.8cm]{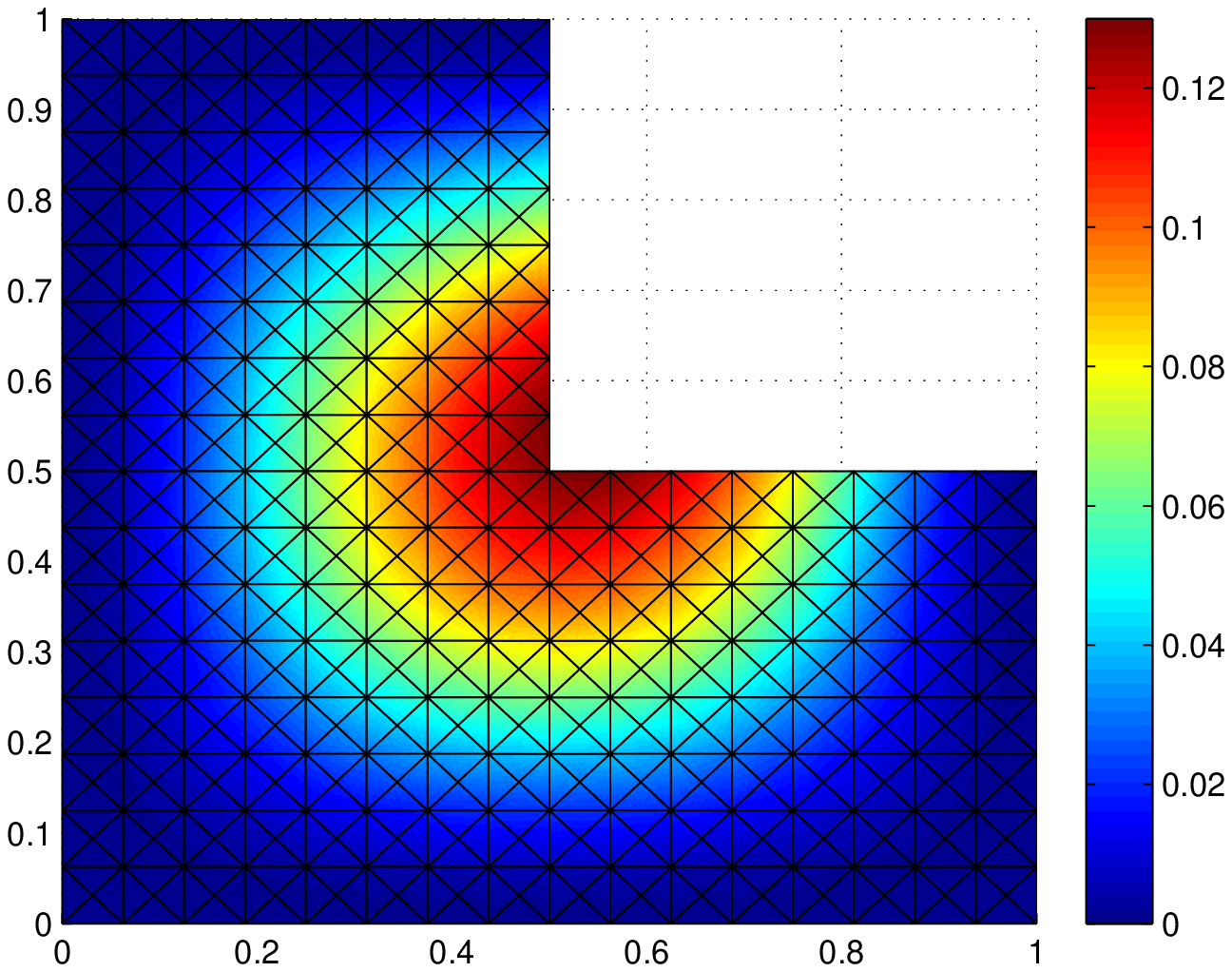}
\end{minipage}
\begin{minipage}{6.3cm}
\centering\includegraphics[height=6cm, width=6.8cm]{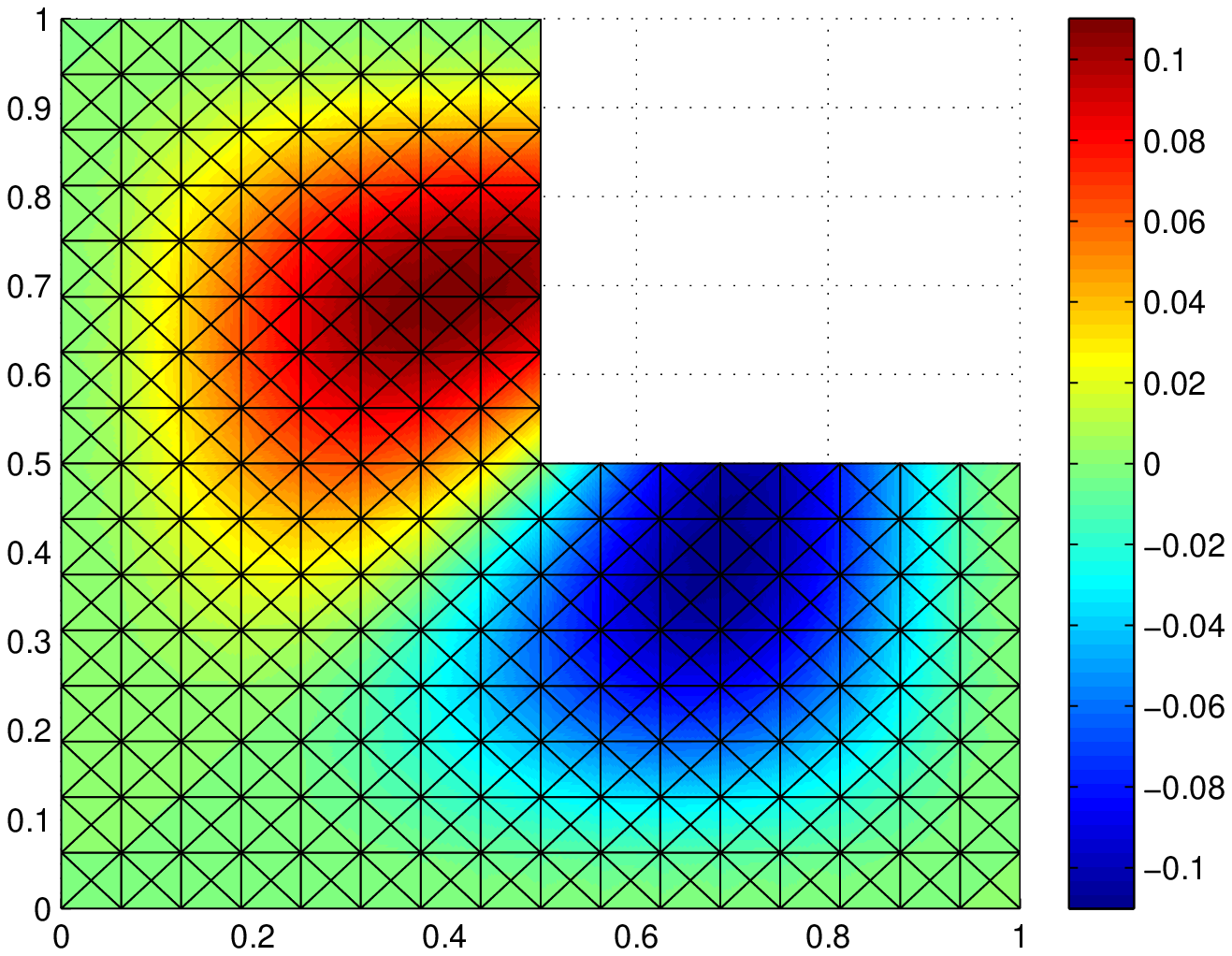}
\end{minipage}
\begin{minipage}{6.3cm}
\centering\includegraphics[height=6cm, width=6.8cm]{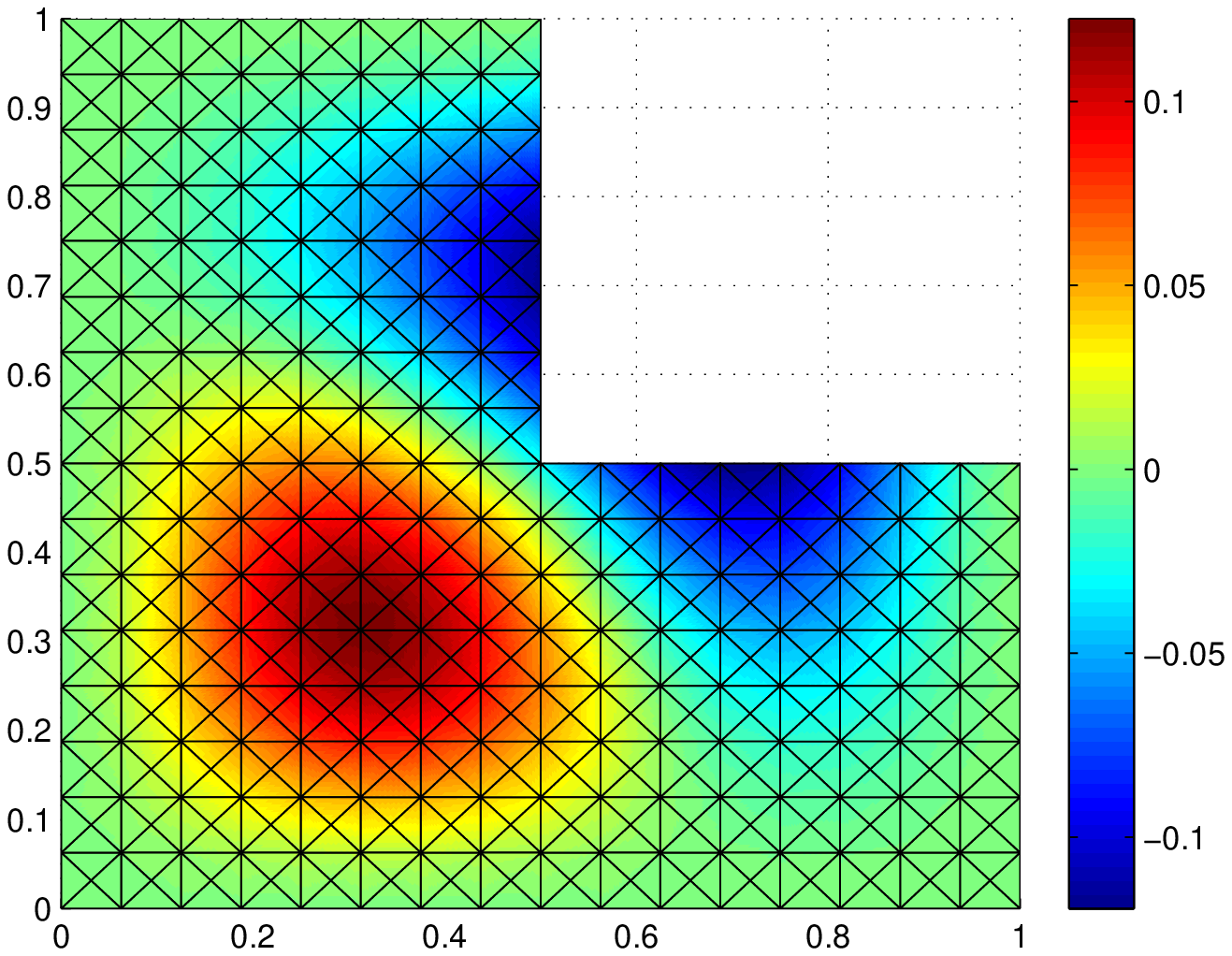}
\end{minipage}
\begin{minipage}{6.3cm}
\centering\includegraphics[height=6cm, width=6.8cm]{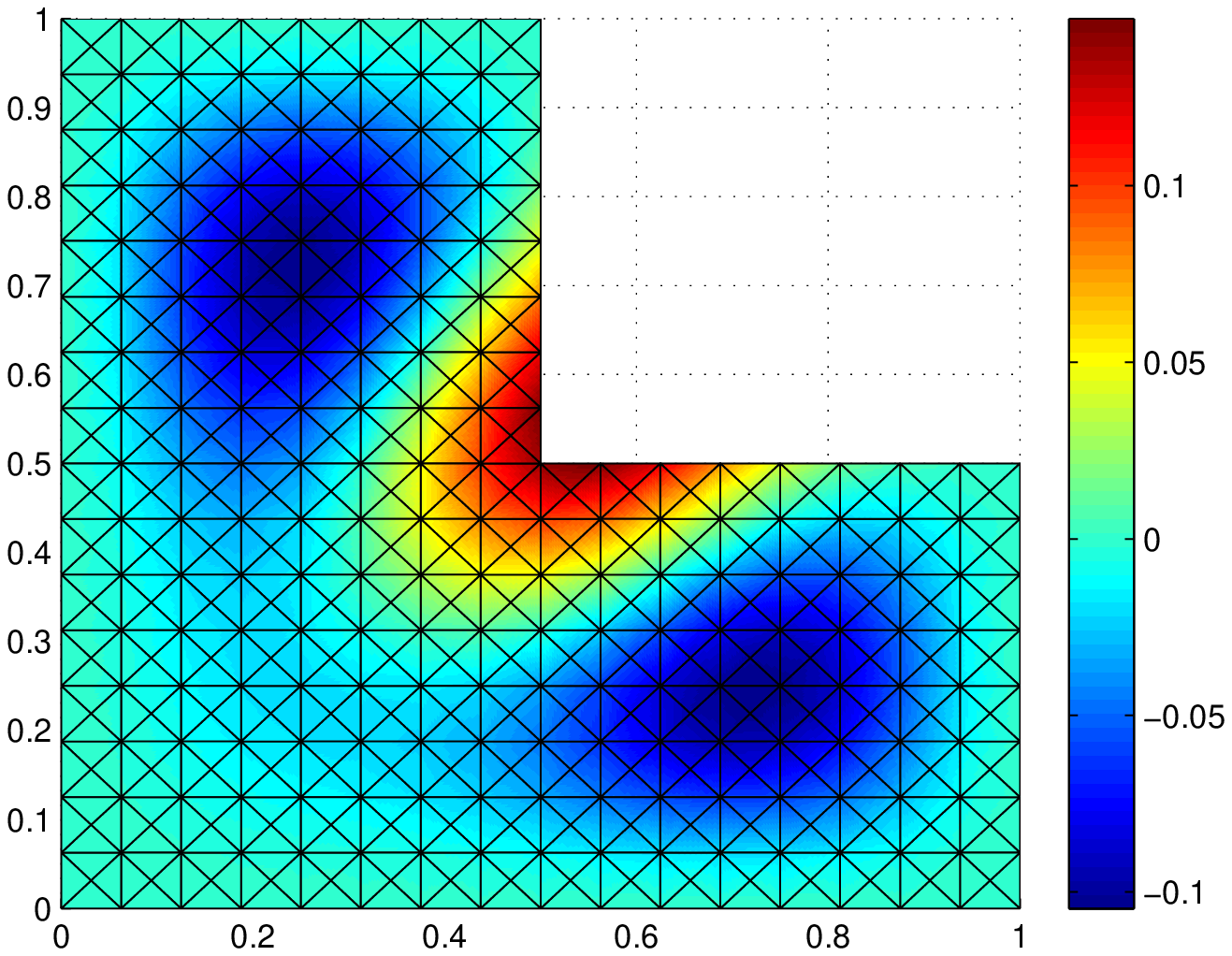}
\end{minipage}
\caption{Eigenfunctions of the plate problem with mixed boundary
condition associated with eigenvalues $\lambda_1$ (top left)
$\lambda_2$ (top right), $\lambda_3$ (bottom left)
and  $\lambda_4$ (bottom right).}
\label{FIG:VM36}
\end{center}
\end{figure}

\section{Conclusions}
\label{SEC:Conclusion}

The mathematical and numerical analysis for the vibration problem
of Kirchhoff-Love plates approximation
by virtual elements was addressed in this paper. The variational formulation
is written in terms of the transverse displacement
of the plate and a conforming $H^2(\Omega)$ discrete formulation
was proposed to numerically approximate the eigenvalue problem.
It is established that the resulting scheme provides a
correct spectral approximation and that the error estimates are
of the optimal order for the eigenfunctions and eigenvalues.
The proposed method is new on triangular meshes,
and in this case the computational cost is almost $3N_v$, where
$N_v$ denotes the number of vertices, thus providing
a very competitive alternative in comparison to other classical
techniques based on finite elements.
The theoretical results obtained were validated numerically.
Even though our theoretical analysis has been developed only
for clamped plates, additional examples have been considered
and we evidenced the results of the previous sections hold
true for more general boundary conditions as well. 


\section*{Acknowledgments}

The first author was partially supported by CONICYT (Chile) through
FONDECYT project 1140791 and by DIUBB through project 151408 GI/VC
Universidad del B\'io-B\'io, Chile.
The second author was partially supported by BASAL Project PFB 03, CMM,
Universidad de Chile, Chile.
%
The third author was partially supported by a CONICYT (Chile)
fellowship.


\bibliographystyle{amsplain}

\end{document}